\documentclass[11pt]{amsart}
\usepackage[dvipsnames]{xcolor}
\usepackage[utf8]{inputenc}
\usepackage{amssymb,amsmath,amsthm,amsfonts,fullpage,float,latexsym,bbm,microtype,cite}
\usepackage{tikz}
\usetikzlibrary{decorations.pathreplacing}
\usepackage{hyperref}
\hypersetup{colorlinks=true, linkcolor=blue, citecolor=magenta, filecolor=magenta, urlcolor=magenta}

\makeatletter
\newtheorem*{rep@theorem}{\rep@title}\newcommand{\newreptheorem}[2]{%
\newenvironment{rep#1}[1]{%
\def\rep@title{\bf #2 \ref{##1}}%
\begin{rep@theorem}}%
{\end{rep@theorem}}}
\makeatother
\newreptheorem{theorem}{Theorem}

\newtheorem*{rep@proposition}{\rep@title}\newcommand{\newrepproposition}[2]{%
\newenvironment{rep#1}[1]{%
\def\rep@title{\bf #2 \ref{##1}}%
\begin{rep@proposition}}%
{\end{rep@proposition}}}
\makeatother
\newreptheorem{proposition}{Proposition}

\newtheorem{theorem}{Theorem}[section]
\newtheorem{proposition}[theorem]{Proposition}

\newtheorem{lemma}[theorem]{Lemma}
\newtheorem{corollary}[theorem]{Corollary}
\theoremstyle{definition}
\newtheorem{remark}[theorem]{Remark}
\newtheorem{definition}[theorem]{Definition}

\newtheorem{example}[theorem]{Example}

\newcommand{\Z}{\mathbb{Z}}
\newcommand{\N}{\mathbb{N}}

\newcommand{\bx}{\mathbf{x}}

\DeclareMathOperator{\val}{val}
\DeclareMathOperator{\cok}{cok}

\DeclareMathOperator{\lcm}{lcm}

\newcommand\commentout[1]{}

\begin{document}
 
  \title{On the critical group of hinge graphs }

\author{Aren Martinian}
\address{Department of Mathematics\\
         University of California, Berkeley}
\email{arenmatinian@berkeley.edu}

\author{Andr\'es R. Vindas-Mel\'endez}
\address{Department of Mathematics, UC Berkeley~\&~Harvey Mudd College\hfill \hfill \phantom{phantom}
\indent \url{https://math.berkeley.edu/~vindas}}
\email{andres.vindas@berkeley.edu ; arvm@hmc.edu}

\begin{abstract}
Let $G$ be a finite, connected, simple graph.
The critical group $K(G)$, also known as the sandpile group, is the torsion subgroup of the cokernel of the graph Laplacian $\operatorname{cok}(L)$. 
We investigate a family of graphs with relatively simple non-cyclic critical group with an end goal of understanding whether multiple divisors, i.e., formal linear combinations of vertices of $G$, generate $K(G)$. 
These graphs, referred to as hinge graphs, can be intuitively understood by taking multiple base shapes and ``gluing" them together by a single shared edge and two corresponding shared vertices. 
In the case where all base shapes are identical, we compute the explicit structure of the critical group.
Additionally, we compute the order of three special divisors. 
We prove the structure of the critical group of hinge graphs when variance in the number of vertices of each base shape is allowed, generalizing many of the aforementioned results.
\end{abstract}

\maketitle

\section{Introduction}
In this paper we study a finite abelian group associated to a finite connected graph $G$, known as the \emph{critical group} of $G$. 
The critical group goes by different names (e.g., the Jacobian group, sandpile group, component group) and is studied in various mathematical areas (e.g., algebraic geometry, statistical physics, combinatorics) \cite{GlassKaplan}.
We focus on the combinatorial definition of the critical group involving chip-firing operations and its connections to graph-theoretic trees. 
In particular, for a finite connected graph, the order of the critical group equals the number of spanning trees of the graph. 
For the interested reader, we recommend the survey paper by Glass and Kaplan \cite{GlassKaplan} as an introduction to the study of critical groups and chip-firing, as well as the books by Klivans \cite{Klivans} and Cory and Perkinson \cite{CoryPerkinson} for comprehensive considerations of chip-firing. 

There are many results on the group structure of the critical group and the relationship with the structure of an associated graph, see for instance \cite{AlfaroValencia, BergetManionMaxwellPotechinReiner,  BrownMorrowZureick-Brown, CoriRossin}. 
Determining the critical group for certain families of graphs continues to be an active area of research. 
There exists work where the critical group has been partially determined for some families of graphs, for instance see \cite{Bai, ChenHou, HouLeiWoo, LiangPanWang, WangPan, WangPanXu, ZhouChen}.
Additionally, there is a growing body of work where the complete critical group structure for families of graphs is determined, see for instance \cite{Biggs, ChenMohar, ChristiansonReiner, JacobsonNiedermaierReiner, Levine, Lorenzini, Lorenzini2, Merris, Musiker, ShenHou, Toumpakari}.

The family of graphs that we study are those which we call \emph{hinge graphs}.
These are graphs that can be intuitively understood by taking multiple base shapes and ``gluing" them together on a single shared edge and two corresponding shared vertices.
In \cite{CoriRossin}, Cori and Rossin show that the critical group of a planar graph $G$ is isomorphic to the critical groups of the dual of $G$.
It so happens that hinge graphs are dual to a family of graphs known as \emph{thick cycle graphs}, which are cycle graphs where multiple edges are allowed, and were studied in \cite{MSRI-UP, Alfaro}.
Furthermore, thick cycle graphs can be seen as specializations of outerplanar graphs studied in \cite{AlfaroVillagran}.

The study of hinge graphs arose independently and was motivated primarily in attempt to answer the question proposed in \cite{GlassKaplan} on how divisors generate critical groups in cases where the group is non-cyclic.
Hinge graphs, especially those containing identical copies of the same base shape, are some of the simplest examples of graphs with non-cyclic critical groups, and whose behavior can be thoroughly studied.
In addition, the study of hinge graphs has led to observations about proving linear equivalence and the order of divisors which provides a streamlined approach to the investigation of divisors and critical groups of graphs more generally.
As mentioned before, hinge graphs can be observed to be the dual graphs of thick cycle graphs, thus the critical group of a hinge graph is isomorphic to the critical group of a thick cycle graph, whose complete structure is given in Theorem 2.29 in \cite{Alfaro} and Theorem 1 in \cite{MSRI-UP}.
Despite there being literature about the structure of the critical groups of thick cycle graphs, and hence an alternate route to the study of hinge graphs, this connection was not previously made.
We provide a novel approach using divisors that generate the critical group of hinge graphs which is not explored in the existing literature. 
We emphasize that the proofs in this paper are new and rely solely on generating divisors and were written with a deliberate choice to avoid using the theory of reduced divisors.

Our main contributions include proofs of formulas for the order and structure of the critical group of hinge graphs with same base shapes, and using tools such as the graph Laplacian, to determine the orders of important group elements.
Using similar techniques to those used to prove these results, we generalize some to the case where we have hinge graphs with different base cycles. 
For what follows $\mathcal{H}_{k,n}$ denotes the hinge graph with $k$ vertices on each base shape and $n$copies of the base shape.
For different cycles, $\mathcal{H}_{k_1-1,k_2-1,\dots, k_n-1}$ denotes the hinge graph with $k_i$ vertices on each base shape.
The critical group of a graph $G$ is denoted $K(G)$. 
\vspace{0.05cm}

\begin{reptheorem}{thm:order}
  Given a hinge graph $\mathcal{H}_{k,n}$, the order of the critical group $K(\mathcal{H}_{k,n})$ is $$|K(\mathcal{H}_{k,n})|=(k-1)^{n-2}(k-1)(k+n-1).$$ 
\end{reptheorem}
\vspace{0.05cm}

\begin{reptheorem}{thm:critical_group}
The critical group $K(\mathcal{H}_{k,n})$ is isomorphic to
$$(\Z/(k-1)\Z)^{n-2}\oplus (\Z/(k-1)(k+n-1)\Z).$$
\end{reptheorem}
\vspace{0.05cm}

\begin{reptheorem}{thm:general_order}
Consider a hinge graph with different base shapes $\mathcal{H}_{k_1-1,\dots,k_n-1}$, the order of\\ $K(\mathcal{H}_{k_1-1,\dots,k_n-1})$ is
\begin{equation*}
|K(\mathcal{H}_{k_1-1,\dots,k_n-1})| = a + a/(k_1-1) + \dots + a/(k_n-1),
\end{equation*}
where $a:=(k_1-1)\cdots(k_n-1)$.
\end{reptheorem}
\vspace{0.05cm}

This paper is organized as follows. 
In Section \ref{sec:prelims} we provide background and preliminaries on divisors, the critical group, and  hinge graphs. 
In Section \ref{sec:samebase} we prove several results for hinge graphs where each base shape is an identical cycle, including the explicit structure of the critical group and the behavior of some noteworthy divisors (see Proposition \ref{lemma:order_of_divisors}).
In Section \ref{sec:different_bases} we generalize many of the aforementioned results to hinge graphs where the number of vertices on each base shape can vary (see Theorem \ref{thm:critical}). 
We conclude in Section \ref{sec:conclusion} with some directions for future research.


\section{Preliminaries}\label{sec:prelims}

In this section we review some key definitions and theorems, as well as introduce several definitions pertaining to our specific case study. 
We begin by briefly detailing divisors and critical groups on graphs.
We take our graphs to be connected, undirected, and any two vertices are connected by at most one edge. 
Note that a consequence of these conditions is that graphs cannot contain loops.
These are the conditions under which much of the existing critical group literature has focused on.

Let $V(G)$ and $E(G)$ refer to the set of vertices and edges of a graph $G$, respectively. 
A cycle graph is a graph where every vertex has valence two.
These are most often thought of as regular polygons (a convention we will adopt).

\begin{definition}
A \emph{divisor} (or \emph{chip configuration}) on a graph $G$ is a formal $\Z$-linear combination of vertices of $G$,
$$D = \sum_{v\in V(G)} v \cdot D(v).$$
The \emph{degree of a divisor} $D$ is the integer $\deg(D):=\sum_{v\in V(G)}D(v)$.
\end{definition}

\begin{definition}
A \emph{firing of a vertex} is the operation taking the divisor $D$ to a divisor $D'$ where
\begin{equation*}
D'(v) = 
    \begin{cases}
        D(v) - \val(v), & \text{if } v = w,\\
        D(v)+ \# \text{ edges between } v \text{ and } w, & \text{if } v \neq w.
    \end{cases}
\end{equation*}
This is referred to as a chip-firing move or chip-firing operation.
We say that two divisors $D$ and $D'$ are \emph{chip-firing equivalent} or \emph{linearly equivalent} if $D'$ can be obtained from $D$ via a sequence of chip-firing moves. 
The \emph{order of a divisor} $D$ is the smallest positive integer $z$ with $zD$ linearly equivalent to the zero divisor.
\end{definition}
In our specific case study, the number of edges between two vertices will always be at most $1$. 
Note also that the chip-firing operation is commutative, and hence firings can be thought of as happening simultaneously.
By convention, if a divisor consists of a 0 associated with a vertex, that vertex is left blank.

The \emph{graph Laplacian} $L$ is defined as $A-M$, where $A$ denotes the adjacency matrix of the graph, and $M$ is the diagonal matrix with the valences of each vertex in $V(G)$.

\begin{definition}
The \emph{critical group} $K(G)$ of a graph $G$ is the torsion subgroup of the cokernel of the graph Laplacian $\cok(L)$.
\end{definition}
\noindent Note that elements of the critical group necessarily have degree $0$, and have order consistent with the definition of order of a divisor.

\begin{theorem}[Corollary 3, \cite{GlassKaplan}]\label{thm:trees}
The order of the critical group $K(G)$ is the number of spanning trees of $G$.
\end{theorem}

While it is sufficient to consider any divisor as an element of the critical group, we are mainly interested in a particular type of divisor that can be thought of as a representative element modulo chip-firing equivalence. 

\begin{definition}
Let $K(G)$ be the critical group of a graph $G$.
We say a divisor $D$ is a $q$-reduced divisor if for any vertex $q$ of $D$,
\begin{enumerate}
    \item $D(v) \geq 0$ for all  $v \neq q$  and 
    \item for every nonempty subset of vertices $V' \subseteq V(G)\setminus \{q\}$, if we take $D(v)$ and fire every vertex in $V'$, then some vertex in $V'$ has associated integer $r < 0$.
\end{enumerate}
\end{definition}
\noindent Note that $q$-reduced divisors are not used explicitly in this paper, but it is worth mentioning that all divisors discussed will be multiples of $q$-reduced divisors.

Next, we introduce definitions specific to a particular family of graphs.

\begin{definition}\label{def:hinge}
A \emph{hinge graph} is a graph constructed by ``adjoining" or ``gluing" several cycle graphs via a shared edge and consequent pair of vertices. 
We call the cycles used to construct the hinge graph the \emph{base shapes}, and will sometimes refer to them as \emph{cycles} of the hinge graph.
\begin{itemize}
  \setlength\topsep{.5em}
  \setlength\itemsep{.5em}
    \item In the case when all base shapes are identical, these hinge graphs are denoted  $\mathcal{H}_{k,n}$, where $k$ is the number of vertices of the base shape (including the pair of shared vertices) and $n$ is the number of copies of the base shape.
    \item For different cycles, we instead use the notation $\mathcal{H}_{k_1-1,k_2-1,\dots, k_n-1}$, where $k_i$ refers to the number of vertices of each base shape. (Note that this notation differs from identical base shapes since we use $k_i-1$ instead of $k_i$ for its usefulness in the proofs of forthcoming results.)
\end{itemize}
\end{definition}

\noindent Refer to Figure \ref{fig:div} for examples of hinge graphs.
In the case of cycles with four vertices, which are taken as squares, these graphs are called \emph{book graphs} as their structure resembles that of several pages. 

In what follows, attaching another copy of the base shape to an existing graph with the same shared hinge will be referred to as \emph{adding a copy} of the base shape.
In cases where spanning trees are counted, the deletion of an edge while retaining both vertices will be referred to as \emph{removing} an edge.

\begin{figure}[ht]
    \centering
    \includegraphics[scale=.4, trim = 0cm 4cm 0cm 4cm, clip]{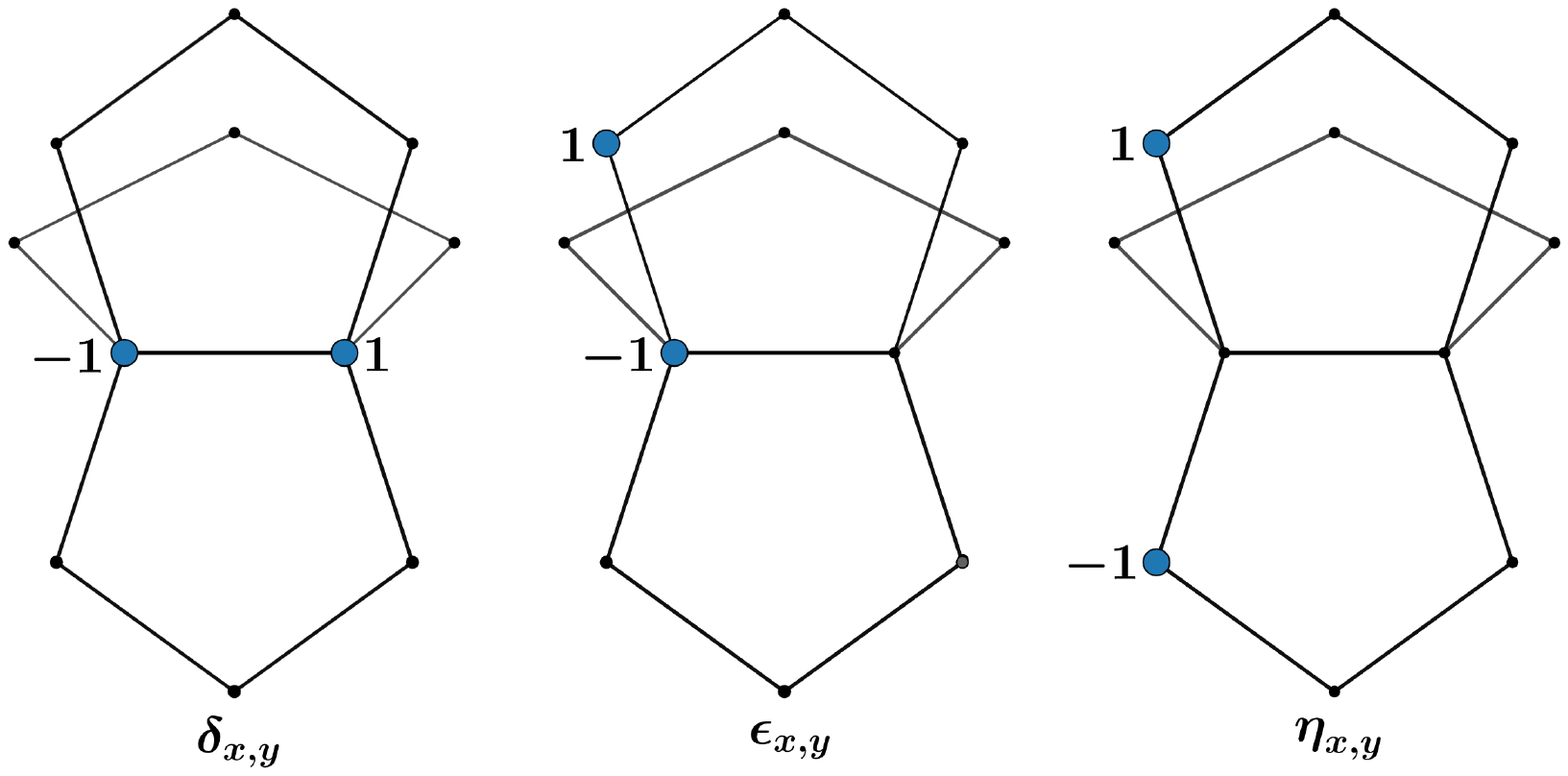}
    \caption{Types of divisors studied: $\delta_{x,y}$ (left), $\epsilon_{x,y}$ (center), and $\eta_{x,y}$ (right).}
    \label{fig:div}
\end{figure}

We will focus our attention on three distinct types of divisors on these graphs, which will be referred to as $\delta_{x,y}$, $\epsilon_{x,y}$, and $\eta_{x,y}$, following the notation in \cite{GlassKaplan}.
All three divisors have degree $0$, with zeroes assigned to all vertices except two, which are assigned a $1$ and $-1$.
The difference between these divisors lies in the location of the vertices associated with nonzero values. The divisor 
\begin{itemize}
  \setlength\topsep{.5em}
  \setlength\itemsep{.5em}
    \item $\delta_{x,y}$ consists of a $1$ and $-1$ on the shared pair of vertices,
    \item $\epsilon_{x,y}$ consists of a $1$ and $-1$ on a shared vertex and an adjacent vertex on a cycle, and
    \item $\eta_{x,y}$ consists of a $1$ and $-1$ on two vertices adjacent to the same shared vertex, but on different cycles.
\end{itemize}  
When enumeration of these divisors is important, as in Lemma \ref{lemma:eta_basis} when choosing a minimal generating set of $\eta_{x,y}$'s or Theorem \ref{thm:diffepsilon} where the order of $\epsilon_{x,y}$ depends on the base shape chosen, we will denote the divisors for each base shape by $\epsilon_{x,y,i}$ and $\eta_{x,y,i}$.
Examples of all three divisors can be seen in Figure \ref{fig:div}.
As we shall see, the distinction of which vertex is assigned a positive or negative value is irrelevant.

Some of the following propositions are well-known in the literature, but we reiterate them here as they will be useful in the forthcoming proofs. 

\begin{proposition}\label{thm:obs1} 
The sum of two divisors (via summing integers at each vertex) corresponds identically to the sum of two elements in the critical group. 
\end{proposition}

We note that the aforementioned statement follows as a consequence of the definition of the critical group, but its impact on our results is significant enough to warrant it as a proposition.

\begin{theorem}[Theorem 11, \cite{GlassKaplan}]\label{thm:kernel}
For any finite connected graph $G$, the null space of the Laplacian
matrix of $G$ is generated by the vector $\mathbf{1}$.
\end{theorem}

As a consequence, borrowing from a vertex is equivalent to firing every other vertex (by adding $\mathbf{1}$ to the vector
$\mathbf{r}=(0,0, \dots 0, -1, 0, \dots 0)^T$, the number of times each vertex is fired), where the $-1$ means borrowing from the $i$th vertex. 
Therefore, borrowing can be thought of as the inverse of firing. 

\begin{proposition}\label{thm:obs2a}
Let $\mathbf{a}$ and $\mathbf{b}$ be two vector representations of divisors for a graph $G$ and let 
$M$ be the augmented matrix of $L$ and $\mathbf{a}-\mathbf{b}$. 
If all entries in $M$ after row reduction are integers, $\mathbf{a}$ and $\mathbf{b}$ are linearly equivalent. 
\end{proposition}

\begin{proof}
It is a widely known fact in linear algebra that augmenting any matrix with a vector is equivalent to solving the vector equation $A\bx=\mathbf{b}$. 
When we employ this construction with the graph Laplacian and augmenting the divisor in the final column, solving this equation is equivalent to finding the vector $\mathbf{r}$, which is the number of times each vertex must be fired to obtain the divisor. 
Since the graph Laplacian has kernel of dimension $1$, by Theorem \ref{thm:kernel}, there will be a last row of zeroes in the matrix, and every value will be expressible in terms of the last column, hence we can simply read off the values in the final column.
If all of these values are integers, we know the divisor is linearly equivalent to the zero divisor, since that means the divisor can be formed by firing vertices an integral number of times. 
Combining this with Proposition \ref{thm:obs1}, two divisors are linearly equivalent if their difference is linearly equivalent to the zero divisor, and hence we can use this method to check if any two divisors are linearly equivalent.
\end{proof}

\begin{proposition}\label{thm:obs2b}
Let $n\mathbf{d}$ be the vector representation of a multiple $n$ of a divisor $\delta$ on a graph $G$, and let $M$ be the augmented matrix of $L$ and $n\mathbf{d}$.
If the non-unital entries of $M$ after row reduction are integers with greatest common divisor $1$, and furthermore if this greatest common divisor is invariant under addition of the vector $\mathbf{1}$, then $n = |\delta|$.
\end{proposition}

\begin{proof}
Using the construction outlined in Proposition \ref{thm:obs2a}, if we multiply the divisor, augment the graph Laplacian with the multiplied divisor, and take its Reduced Row Echelon Form, we will obtain the corresponding vector $\mathbf{r}$, unique up to addition of $\mathbf{1}$. 
Note that $\mathbf{r}$ will only have integer entries if the divisor is linearly equivalent to the zero divisor.
Through this process, if we obtain a set of integers whose greatest common divisor equals $1$, and furthermore adding a multiple of $\mathbf{1}$ does not change this greatest common divisor, we have discovered the smallest multiple of the divisor for which we have linear equivalence to the zero divisor.
This is because if there is no common divisor, dividing by any integer leaves us with fractional components of $\mathbf{r}$, and adding or subtracting multiples of $\mathbf{1}$ does not change this fact.
Therefore, we can use this result to prove the order of any divisor.
\end{proof}

As the following lemma is useful to prove results for hinge graphs with different base shapes $\mathcal{H}_{k_1-1, \dots k_n-1}$, we consider $k_i-1$ rather than $k_i$.

\begin{lemma}\label{lemma:bgcd}
Let $b=\lcm(k_1-1\dots, k_n-1)$.
Then $$\gcd(b/(k_1-1),\dots, b/(k_n-1)) = 1.$$
\end{lemma}

For completeness, we detail the proof below.
\begin{proof} 
Let $$g:=\gcd(b/(k_1-1),\dots, b/(k_n-1)).$$ 
Then $g$ divides $b/(k_j-1)$ for all $1 \leq j \leq n$. 
Hence, for each $j$ there exists some $d_j$ such that $$g\cdot d_j=b/(k_j-1),$$ so $g\cdot d_j(k_j-1)=b$ and $g$ divides $b$.
Dividing through by $g$ we see that $$(k_j-1)(d_j)=b/g,$$ and therefore $(k_j-1)$ divides $b/g$, but this was true for all $j$ and hence $b/g$ must be a multiple of all $k_j-1$. 
Since $b$ is the least common multiple that means $g=1$. 
\end{proof}

\begin{remark}
This lemma holds if we eliminate one of the $k_i-1$ provided that the least common multiple does not change, which we make use of in the proof of Theorem \ref{thm:diffepsilon}.
\end{remark}

We use the above lemma to determine that the number of times that vertices are fired for our specific divisors, multiplied by their orders (e.g., $(k-1)\eta_{x,y}$), are coprime. 
Therefore, we obtain that specific multiple as the order of the divisor in the critical group.
However, this is not sufficient, as this coprimality between our elements may not be invariant under addition of the vector $\mathbf{1}$, the kernel of the graph Laplacian. 
We will use the next result in the proofs of Proposition \ref{diffdelta} and Theorem \ref{thm:diffepsilon} to ensure this is not the case. 

\begin{lemma}\label{thm:invariantgcd}
Let $\{n,2n,...mn\}$ and $\{m,2m,...nm\}$ be two sets of positive integers with $n,m$ coprime. Then there exists some element of each set such that they differ by exactly 1. 
\end{lemma}


\section{Hinge graphs with the same base shapes}\label{sec:samebase}

In this section, we prove several theorems on the behavior of hinge graphs when all base shapes are identical.

\begin{theorem}\label{thm:order}
  Given a hinge graph $\mathcal{H}_{k,n}$, the order of $K(\mathcal{H}_{k,n})$ is $$|K(\mathcal{H}_{k,n})|=(k-1)^{n-2}(k-1)(k+n-1).$$ 
\end{theorem}

\begin{proof}
We proceed by induction on $n$, the number of copies of the base shape.
As the base case, we consider when $n=2$, that is, we have the hinge graph with two cycles.
By Theorem 3.1 in \cite{BeckerGlass} we garner that the critical group of this hinge graph has order $k^2-1$. 

Next, we consider a hinge graph with $n$ copies of the base shape and then add an extra copy, that is, we begin with $\mathcal{H}_{k,n}$ and consider the hinge graph $\mathcal{H}_{k,n+1}$. 
In what follows, we denote the number of spanning trees of $\mathcal{H}_{k,n}$ by $S(n,k)$. 

\begin{figure}[ht]
    \centering
    \includegraphics[scale=.25, trim = 4cm 3.5cm 15cm 3cm, clip]{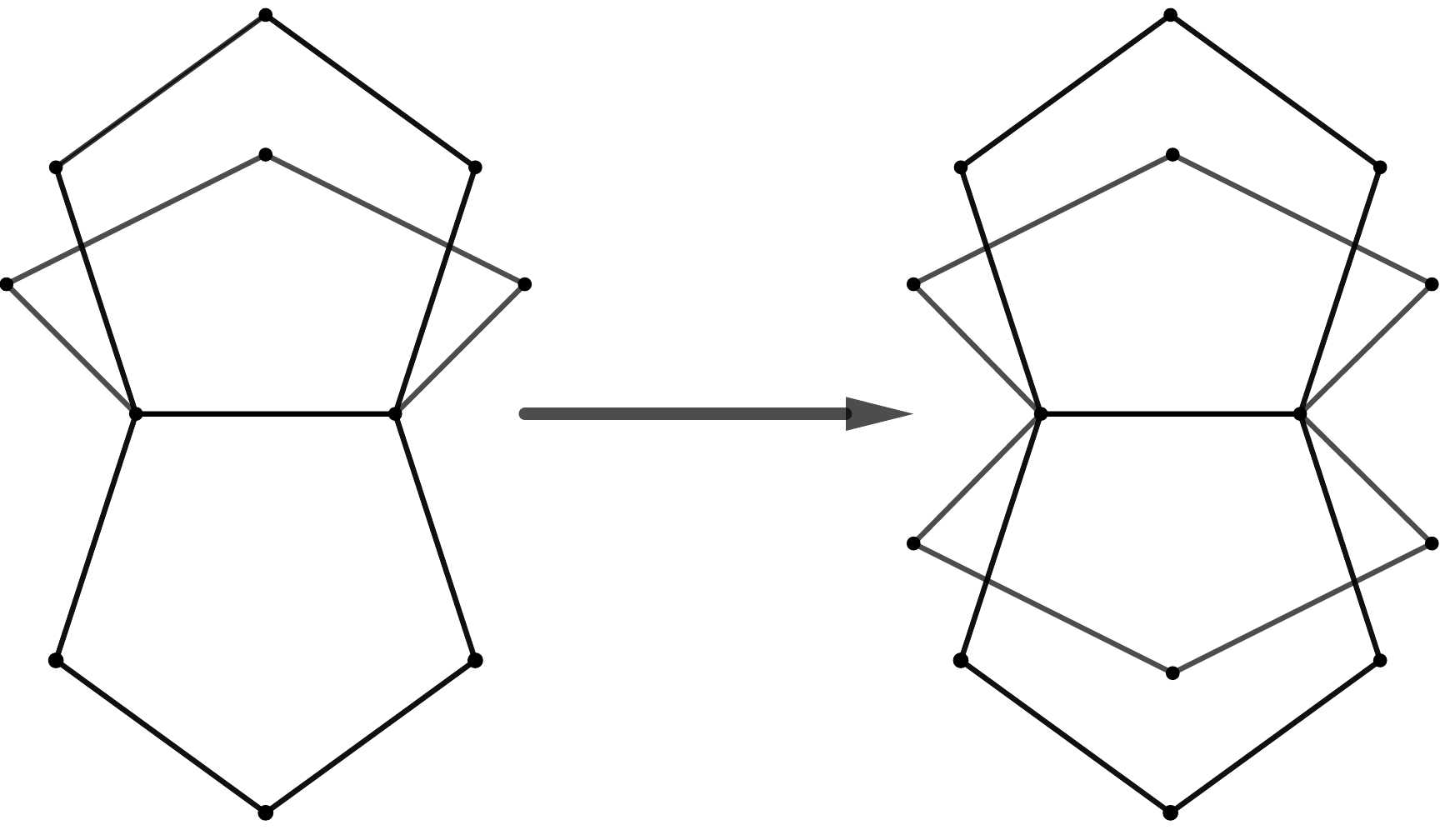}
    \caption{Adding a copy of the base shape to $\mathcal{H}_{5,3}$ to create $\mathcal{H}_{5,4}$.}
    \label{fig:baseshape}
\end{figure}

Whenever another copy of the base shape is added, the number of edges and vertices increases by $k-1$ and $k-2$, respectively.  
Recall that for a spanning tree $T$ of a graph $G$ we have that $|E(T)| = |V(T)|-1$. 
Assuming the condition was previously met, we have the equality 
$$|E(T)|+k-1 = |V(T)|-1+k-2.$$
It then follows that $|E(T)| = |V(T)|-2$,
and therefore for the condition of a spanning tree to be met, we must remove an additional edge.

Note that we cannot remove any two edges on the cycle unless one of them is the shared edge.
This is because subtracting two edges from the same cycle forces the graph to be disconnected for precisely the same reason that removing two edges from a single cycle disconnects the graph (see Figure \ref{fig:theorem3.1}).
On the other hand, the shared edge can be removed since both vertices have valence greater than 2, so neither vertex will be isolated upon its removal.
Thus, there are two possibilities for selecting which edge to remove: one additional edge from the ($n+1$)st base shape or the shared edge. 

\begin{figure}[ht]
    \centering
    \includegraphics[scale=.3, trim = 1.5cm 1.5cm 5.5cm 2.5cm, clip]{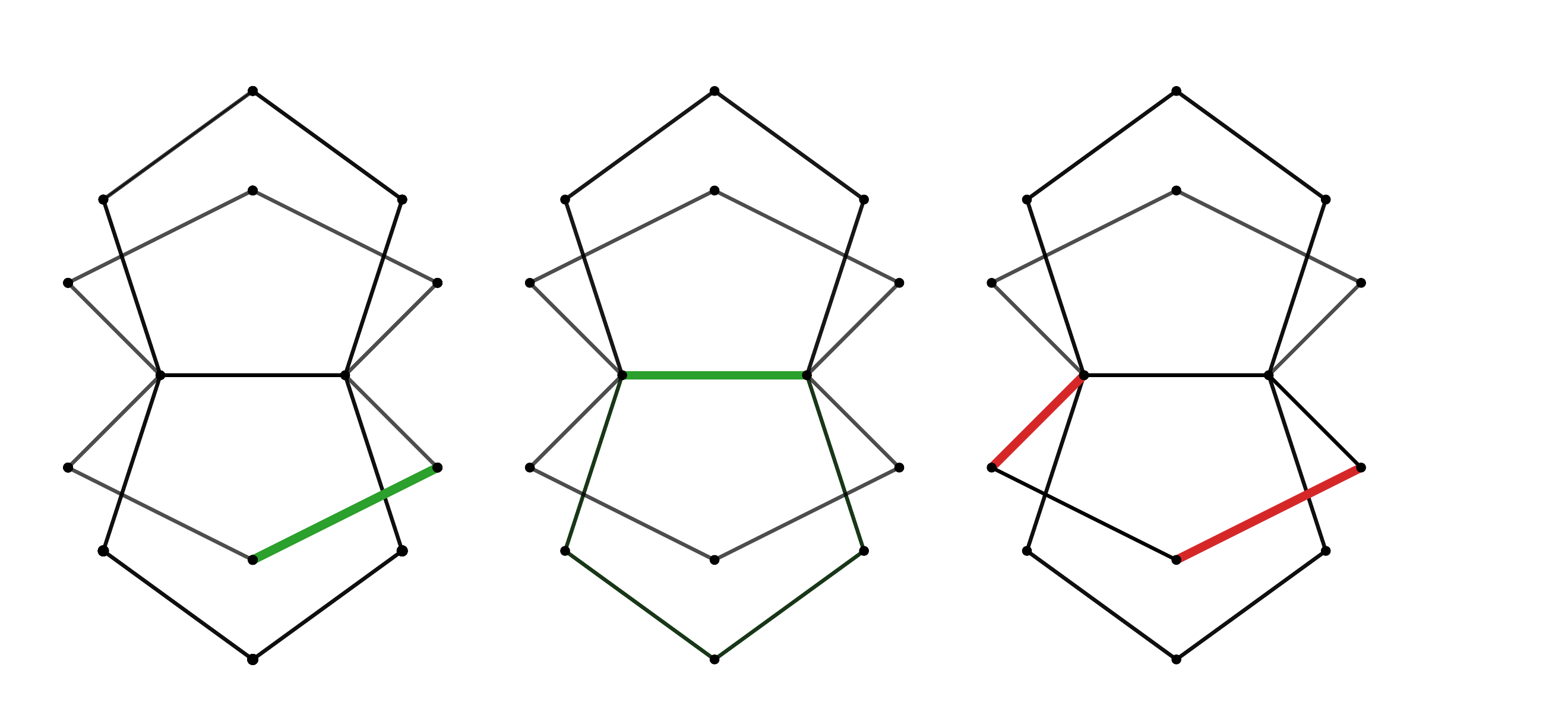}
    \caption{Removing edges to create spanning trees. (Left) removing edges from the new cycle gives us $(k-1)S(n,k)$ possibilities (example of removable edge in green). (Middle) removing the shared edge (in green) gives us $(k-1)^n$ extra possibilities. (Right) we cannot subtract two edges from the same cycle, since that disconnects the graph (example of two non-removable edges in red).}
    \label{fig:theorem3.1}
\end{figure}

We have $k-1$ possible edges to remove from the new ($n+1$)st cycle, and with each option we have $S(n,k)$ spanning trees since we can remove edges in precisely the same way as with $n$ base shapes.
Hence, we obtain $(k-1)(S(n,k))$ spanning trees.

The other option is to the subtract from the shared edge, leaving us to choose one edge from each of the $n$ original cycles. 
Therefore, we have an additional $(k-1)^n$ possible ways to remove the $n$ edges.

To prove that $|\mathcal{H}_{k,n+1}|$ is obtained by the desired formula, we show algebraic equivalence.
This is sufficient as, by Theorem \ref{thm:trees}, the order of the critical group $|\mathcal{H}_{k,n+1}|$ is equivalent to $S(n+1,k)$, the number of spanning trees of $\mathcal{H}_{k,n+1}$.
We proceed as follows:
first, we write down the formula for $n$ copies, apply the operations described above, and then show this is algebraically equivalent to the formula for $n+1$ copies: 
\begin{align*}
(k-1)^{n-2}(k^2-1+(n-2)(k-1))&=(k-1)^{n-2}(k-1)(k+1+n-2)\\
&=(k-1)^{n-2}(k-1)(k+n-1).\\
\end{align*}
Now we apply the recursive operation:
\begin{align*}
(k-1) \left((k-1)^{n-2}(k-1)(k+n-1)\right)+(k-1)^n&=(k-1)^{n-1}((k-1)(k+n-1)+k-1)\\
&=(k-1)^{n-1}(k^2-1+(n-1)(k-1)). \\
\end{align*}
This is the formula one obtains by replacing $n$ with $n+1$ in the original. 
Thus, by Theorem \ref{thm:trees}, we have proved the order of the critical group.
\end{proof}

\begin{remark}
In Theorem \ref{thm:general_order} we generalize Theorem \ref{thm:order} in the setting where our hinge graph has base shapes with different number of vertices, i.e., different cycles. 
\end{remark}
\begin{proposition}\label{lemma:order_of_divisors}
The orders of the divisors $\eta_{x,y}$, $\delta_{x,y}$, and $\epsilon_{x,y}$ are $k-1$, $k+n-1$, and $(k-1)(k+n-1)$, respectively.
\end{proposition}

We prove the orders of the divisors independently.
First, we illustrate a sketch of each proof for readability:

\begin{proof}[Sketch:]\hfill 
\begin{itemize}
  \setlength\topsep{.5em}
  \setlength\itemsep{.5em}
\item To prove the order of $\eta_{x,y}$, we begin with the hinge graph $\mathcal{H}_{3,n}$ and describe an iterative procedure on $k$, the number of vertices of the base shape, to obtain order of $\eta_{x,y}$ for $\mathcal{H}_{k,n}$ for arbitrary $k$ and $n$. 
We consider a cycle as a string of $k$ vertices whose endpoints are the vertices on the shared edge, which we can do because the shared endpoints are not fired.
Each time we add a vertex to increase $k$, we apply a firing operation on all vertices except the endpoints to achieve linear equivalence to the zero divisor.
Concatenating this procedure each time we add a vertex, we see we have fired each non-shared vertex a consecutive number of times and hence this must be the order since any two consecutive numbers are coprime.

\item For $\delta_{x,y}$, we apply the same procedure as $\eta_{x,y}$, but note we start a vertex further away from the endpoint and also need to consider all $n$ base shapes, rather than just a single one.

\item Lastly, the proof of $\epsilon_{x,y}$ combines these two proofs, first showing that a scalar multiple of $\epsilon_{x,y}$ is equivalent to $\delta_{x,y}$ and then using the procedure from $\delta_{x,y}$. 
In this section, all three proofs rely on the consecutive number of firings of vertices. 
\end{itemize}
\end{proof}
\vspace{0.25cm}

\begin{proof}[Proof of order of $\mathbf{\eta_{x,y}}$]
We begin with the smallest case, that is the case where all base shapes are triangles. 
Then we will describe an iterative procedure to obtain all base shapes with more than three vertices.
If we assign a $1$ and $-1$ to the non-shared vertex of two different triangles, it is straightforward to show this divisor is a group element of order at most $2$.

\begin{figure}[ht]
    \centering
    \includegraphics[scale=.25, trim = 5.5cm 1.5cm 8.5cm 3.5cm, clip]{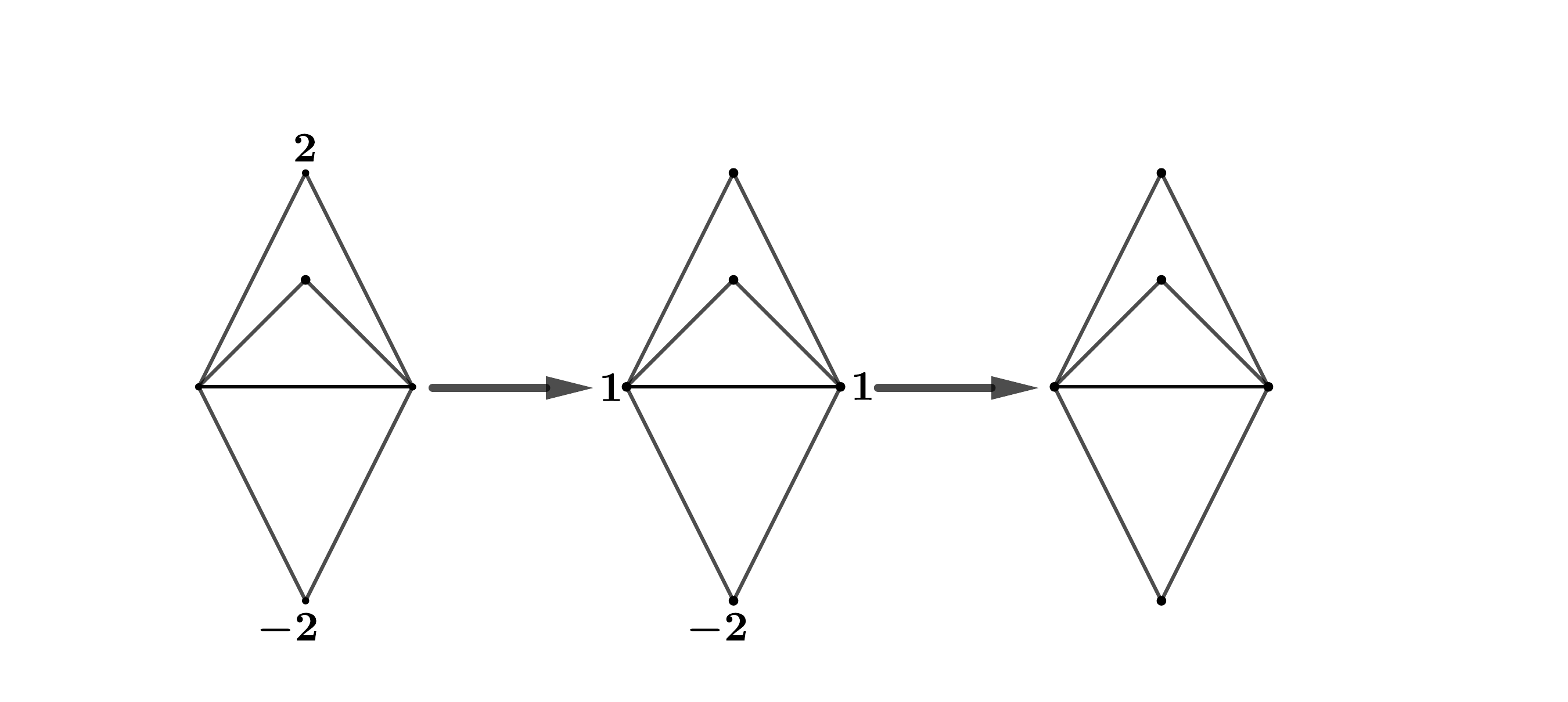}
    \caption{The base case, where linear equivalence can be seen by firing the vertex with a 2 and borrowing from the vertex with -2.}
    \label{fig:Basecase}
\end{figure}

As shown in Figure \ref{fig:Basecase}, we multiply this divisor by $2$ and hence have $2$ and $-2$.
Since we can fire the $2$ exactly once and borrow from the $-2$ exactly once to obtain linear equivalence to the zero divisor, we have shown that the divisor is a group element of at most order $2$.
It is possible that the divisor is linearly equivalent to the zero divisor without any scalar multiplication, which would mean it would be of order $1$. 
One can verify this is not true by utilizing Proposition \ref{thm:obs2b}, which tells us that chip firing equivalence is unique up to addition of the kernel (which in this case is simply $\mathbf{1}$).
From this, if we fire each vertex a number of times, such that the greatest common divisor of the firings is $1$ irrespective of addition of the vector $\mathbf{1}$, then we immediately obtain the order of the divisor. 
This is true because a smaller order, and hence a smaller multiple of each value, would give us fractional firings which cannot be changed by adding integer multiples.
Therefore, we have that for triangles, the order of $\eta_{x,y}$ is $k-1=3-1=2$.

\begin{figure}[ht]
    \centering
    \includegraphics[scale=.3, trim = 3cm 2cm 30cm 2.5cm, clip]{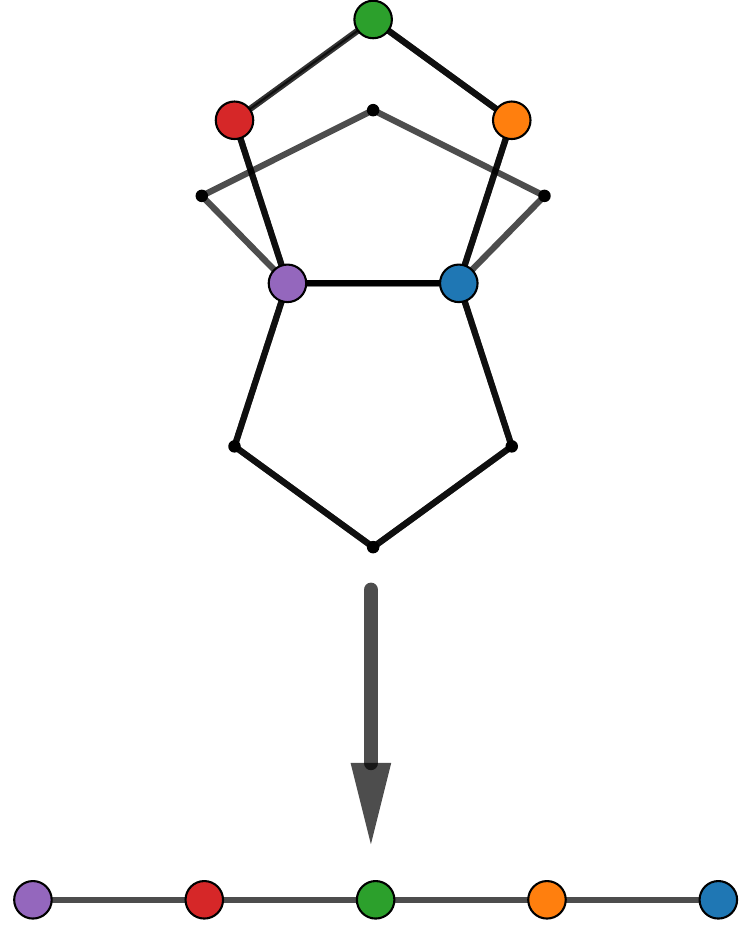}
    \caption{To each vertex of one cycle, we assign a vertex on the string of vertices.}\label{fig:assign}
\end{figure}

Now consider one of the base shapes of $\mathcal{H}_{k,n}$.
From this shape, we associate with it a string of vertices obtained by removing the shared edge as depicted in Figure \ref{fig:assign}.
As shown, the shared vertices correspond with the endpoints of the string. 
Even though they connect via an edge on the base shape, this does not pose a problem as we do not fire either vertex.

Note that each vertex in the string has valence $2$, except for the endpoints, which have valence $1$.
If we fire every vertex except the endpoints, we obtain the value $1$ associated with each endpoint, the value $-1$ associated with the vertices adjacent to each endpoint, and $0$ everywhere else, as depicted in Figure \ref{fig:chain}.
This follows because each time we fire a vertex that is not an endpoint or adjacent to an endpoint, the integer associated with the vertex decreases by two, but firing both of its neighbors returns the two chips lost for a net effect of zero.

\begin{figure}[ht]
    \centering
    \includegraphics[scale=.4, trim = .5cm 6cm 26cm 6cm, clip]{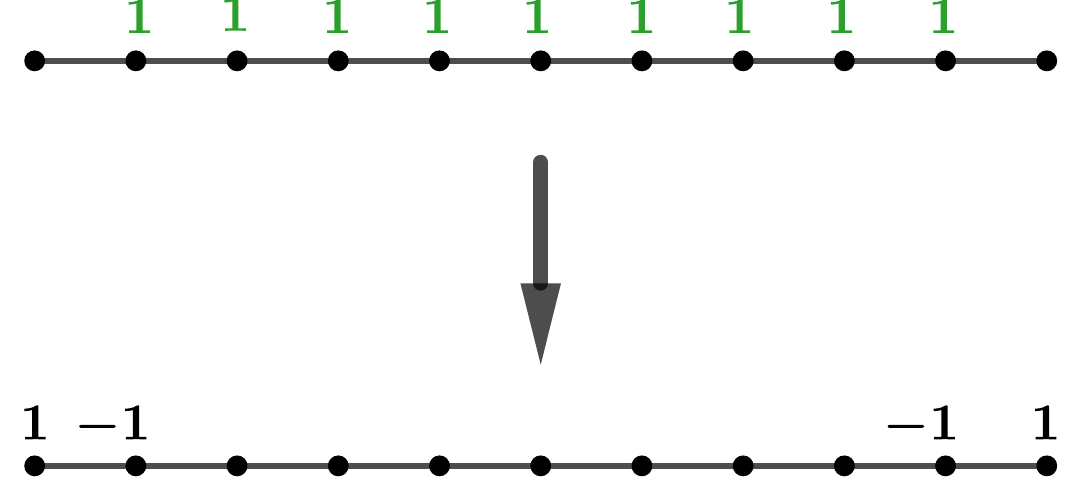}
    \caption{When firing all vertices in a chain except the endpoints exactly once, as indicated by the green numbers, we have that the amount each vertex changes by is $(1, -1, 0,\dots,0, -1, 1)$.}
    \label{fig:chain}
\end{figure}

The result is twofold: given an existing configuration of $1$s associated with endpoints and $-1$s associated with vertices adjacent to them, adding an additional vertex and repeating the process ``pushes" the $1$ an additional vertex farther. 
Alternatively, if we decide not to add an additional vertex, we must compensate by increasing the value associated with the vertex adjacent to the endpoint by $1$.

Since we are not firing the endpoints, we can apply the aforementioned concept to a cycle, with the endpoints being the vertices of the shared edge, as illustrated in Figure \ref{fig:fireprocess}.

\begin{figure}[ht]
    \centering
    \includegraphics[scale=.3, trim = 1cm 5cm 7cm 4cm, clip]{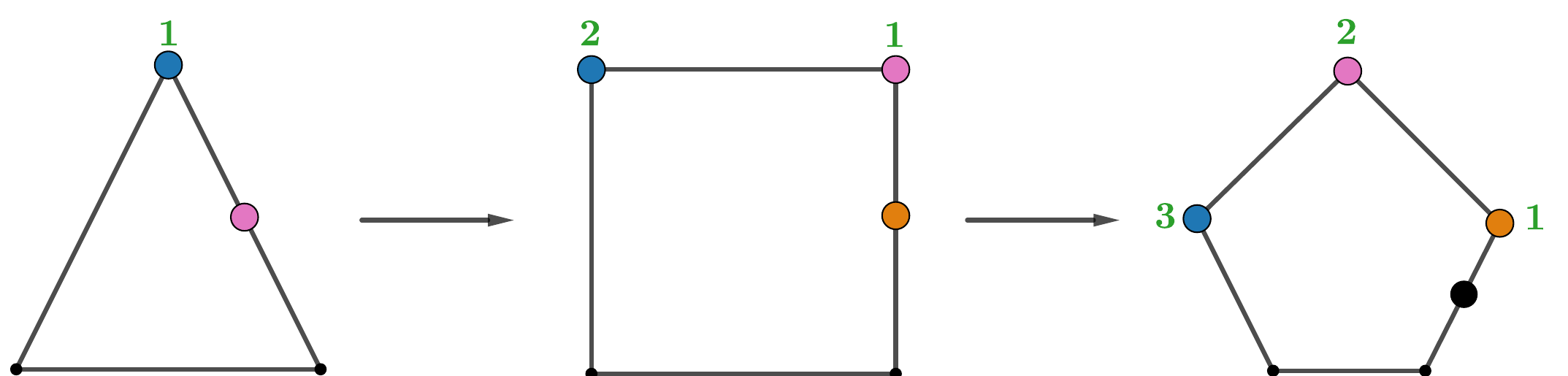}
    \caption{The chip firing process applied iteratively to the triangle base shape. The green numbers represent the number of times each vertex on the graph is fired. This results in the 1 ``pushed" an additional vertex further, as shown in Figure \ref{fig:consecfiring}.  }\label{fig:fireprocess}
\end{figure}

In the cycle configuration, each time we increase $k$ by $1$, we must extend the string of vertices on the right by $1$ vertex, but leave the left unchanged.\footnote{Note that the choice of left and right is irrelevant, but for the sake of consistency we will refer to the nonzero values of the divisor as being situated on the left.}
Each time we introduce a new vertex we must fire every vertex exactly once, except the shared vertices, and also add a $+1$ to the value associated with the vertex adjacent to the left shared vertex. 
For example, enumerating the non-shared vertices of the cycle (from left to right),
\begin{enumerate}
  \setlength\topsep{.5em}
  \setlength\itemsep{.5em}
    \item Triangle: Fire 1 Once.
    \item Square: Fire 1 Twice, Fire 2 Once. 
    \item Pentagon: Fire 1 Thrice, Fire 2 Twice, Fire 3 Once.
\end{enumerate}
And so on. 
Refer to Figure \ref{fig:consecfiring}.\footnote{We will refer to this procedure as the ``consecutive chip-firing process" in similar proofs.}

\begin{figure}[ht]
    \centering
    \includegraphics[scale=.3, trim = 6cm 1cm 21.5cm 0.5cm, clip]{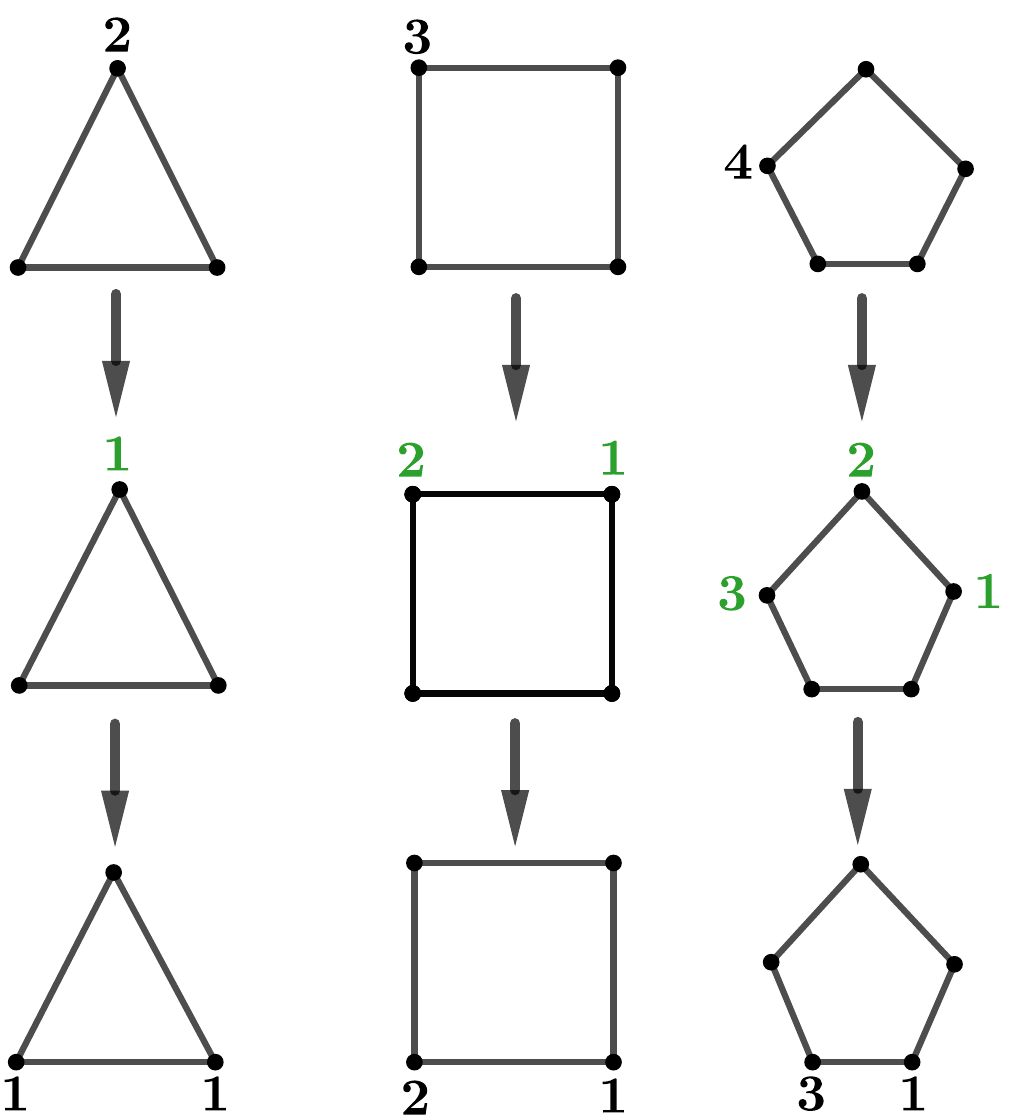}
    \caption{The consecutive chip-firing process applied to the triangle, square, and pentagon base shapes. The top row illustrates the positive component of $\eta_{x,y}$ on the base shapes. In the middle row the green numbers represent the number of times each vertex is fired, a result of subdividing the edge with another vertex and requiring all vertices but those on the shared edge to be fired an additional time. The bottom row represents the divisor resulting from each chip-firing operation.}\label{fig:consecfiring}
\end{figure}

By Theorem \ref{thm:kernel}, we can execute the identical process on the negative integer which will  cancel out the nonzero values on the shared vertices.
Since we begin with $2$ for the triangle, and add $1$ for each additional vertex added, we obtain $k-1$ as the value associated with the vertex adjacent to the left-shared vertex.
For a multiple of $\eta_{x,y}$ to be linearly equivalent to the zero divisor it must be at least $(k-1)\eta_{x,y}$. 
The order must be $k-1$ since $k-1$ and $k-2$ are coprime invariant under $\mathbf{1}$. 
Using Proposition \ref{thm:obs2b}, this completes the proof.
Hence, $\eta_{x,y}$ is an order $k-1$ element of the critical group. 
\end{proof}
\vspace{0.25cm}

\begin{proof}[Proof of order of $\mathbf{\delta_{x,y}}$]\label{pf:order_delta}
The proof for $\delta_{x,y}$ follows an identical procedure as the proof for $\eta_{x,y}$. 
The two major differences, however, are that we need to apply the consecutive chip-firing process $n$ times, once for each base shape, rather than just on a single cycle, and that instead of firing the vertex adjacent to the shared edge, we are firing a vertex on the shared edge. 
This means the order of the divisor is $$(k-1)+1+(n-1) = k+n-1,$$ where the additional $+1$ comes from the shift in position, and the additional $n-1$ comes from this process occurring simultaneously on the $n-1$ extra cycles. 

This also follows from the number of times that vertices adjacent to the negative shared vertex are fired.
Since we want to apply the consecutive chip-firing process to each of the $n$ base shapes, we require each adjacent vertex on the cycles to fire once.
Following the vertices along the base shapes to the other shared vertex, we see that it must be fired $k-1$ times, hence the order is $k+n-1$.
\end{proof}
\vspace{0.25cm}

\begin{proof}[Proof of order of $\mathbf{\epsilon_{x,y}}$]
To prove that the order of $\epsilon_{x,y}$ is $(k-1)(k+n-1)$, we cannot simply show that $(k-1)$ multiplied by each vertex is linearly equivalent to $\delta_{x,y}$, because it is not true for a group $G$, $a,b\in G$, and $c\in \Z$ that  $(c)|a|=|b|$ when $cb=a$.
Instead, we can utilize Theorem \ref{thm:obs2b} and {\cite[Theorem 1]{GlassKaplan}}, which tells us that chip firing equivalence is unique up to addition of the kernel (which is simply $\mathbf{1}$). 
As above in the case of $\eta_{x,y}$, if we can find two vertices with a coprime number of firings, we have deduced the order of $\epsilon_{x,y}$.

Finding the order of $\epsilon_{x,y}$ when the base shapes are all the same first involves firing a multiple of $\epsilon_{x,y}$  to be equivalent to $\delta_{x,y}$.

To show that $(k-1)\eta_{x,y}$ is linearly equivalent to $\delta_{x,y}$ (as in the proof of $\eta_{x,y}$) we can associate the vertices on the cycle with a string of vertices.
Now, we can apply the consecutive chip-firing process to $(k-1)\eta_{x,y}$.
On the left endpoint, we initially have $(k-1)(-1)$, but we fire the adjacent vertex $k-2$ times, hence the left endpoint obtains a value of $-(k-1)+k-2=-1$. 
By firing each successive vertex a consecutive number of times, we ensure that all vertices except the endpoints are $0$ as in the proof of $\eta_{x,y}$, and we are left with a $1$ on the right endpoint.
But this aligns with $\delta_{x,y}$, since the endpoints correspond to the shared vertices. 
See Figure \ref{fig:convert}.

\begin{figure}[ht]
   \centering
   \includegraphics[scale=.25, trim = 2.5cm .5cm 13cm .5cm, clip]{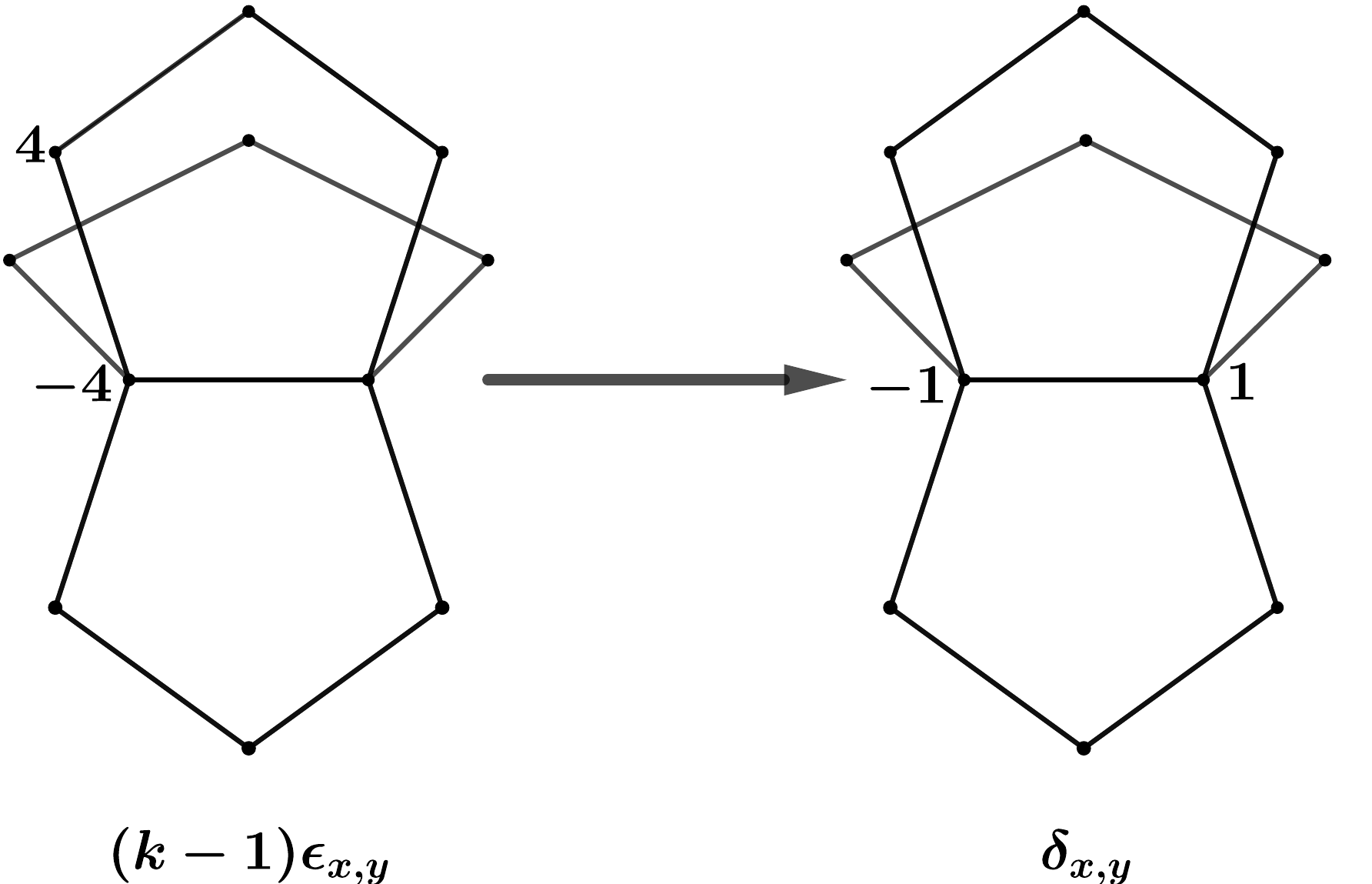}
   \caption{The first step of determining $|\epsilon_{x,y}|$, showing $(k-1)\epsilon_{x,y}$ is linearly equivalent to $\delta_{x,y}$ illustrated on the graph $\mathcal{H}_{5,3}$. Recall that the $4$ becomes $3$ and $1$ on the shared vertices via the chip-firing process, and consequently the only nonzero integers on vertices are $1$ and $-1$.}
   \label{fig:convert}
\end{figure}

Next in order to obtain linear equivalence to a multiple of $\delta_{x,y}$, we can use Theorem $\ref{thm:obs1}$ which tells us that given a chip firing configuration, firing a multiple $m$ of $\mathbf{r}$ (where $\mathbf{r}$ is defined as in the proof of Theorem \ref{thm:obs2a}) is equivalent to repeating the chip-firing process $m$ times, and therefore directly corresponds to multiplying the values associated with each divisor by $m$.  
Instead of utilizing the consecutive chip-firing process as in the proof of $\eta_{x,y}$, we can multiply the entire process by $k+n-1$, which gives us linear equivalence to $(k+n-1)(\delta_{x,y})$ as opposed to simply $\delta_{x,y}$, as shown in Figure \ref{fig:epsilon}.

\begin{figure}[ht]
   \centering
   \includegraphics[scale=.35, trim = 1cm 3cm 3cm 1cm, clip]{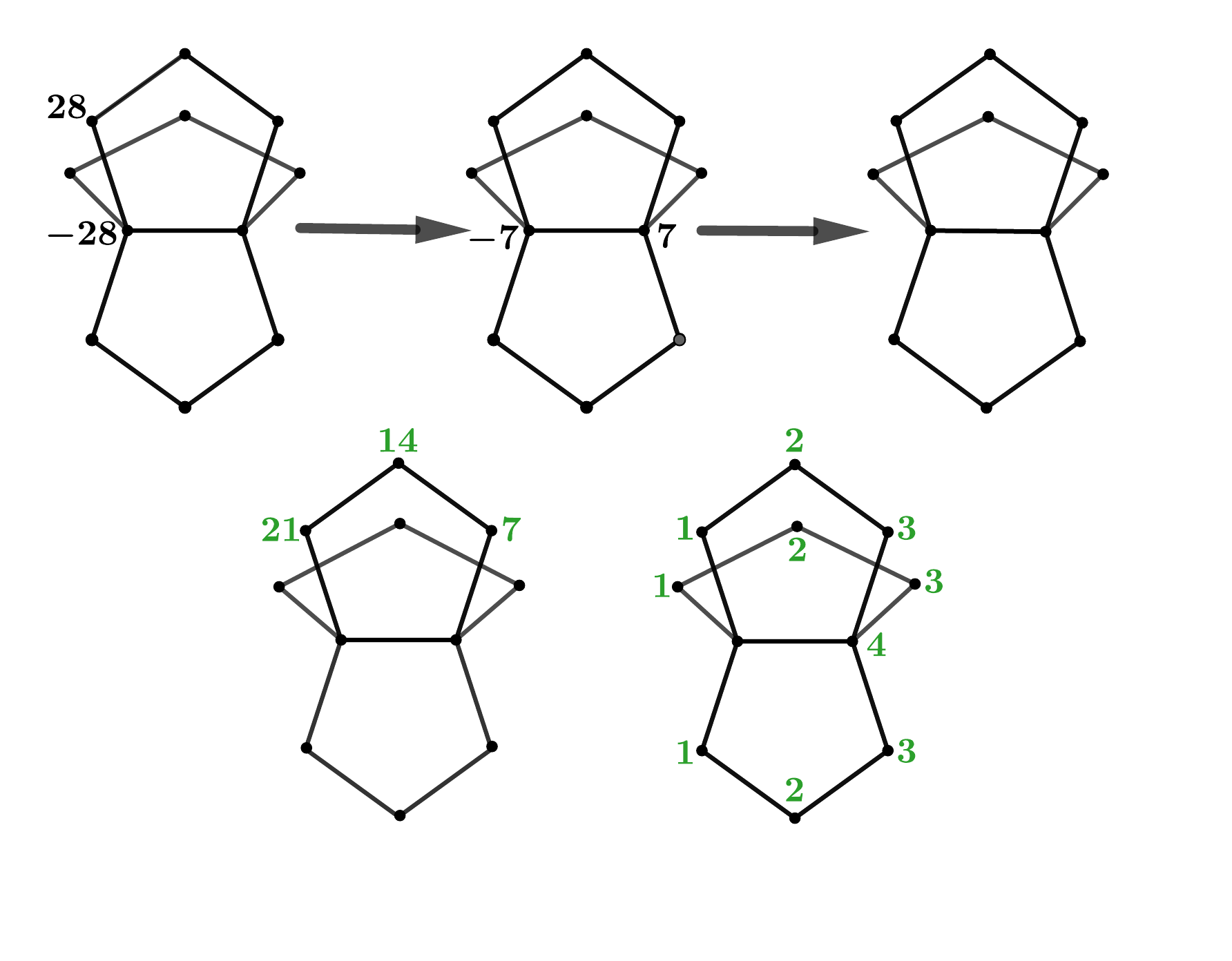}
   \caption{
   Top: $(k-1)(k+n-1)\epsilon_{x,y}$ is linearly equivalent to $(k+n-1)\delta_{x,y}$, which is then linearly equivalent to the zero divisor. Bottom: the green numbers represent the number of times each vertex is fired to obtain equivalence. In the bottom left, these green numbers are $k+n-1 = 7$ times the consecutive chip-firing process, as we have multiplied the divisor by $7$.}
   \label{fig:epsilon}
\end{figure}

Since our hypothesis is that the order is $(k+n-1)$ times as large, instead of firing a consecutive number of times along the cycle, we would fire multiples of $k+n-1$.
This works as we are working under the integers, and hence multiplying the number of firings consequently multiplies the effect of the firings.
But we determined that the order of $\delta_{x,y}$ is $k+n-1$, and hence for any cycle excluding the one initially containing the positive integer of $\epsilon_{x,y}$ we see a consecutive number of firings: $k-2, k-3, \dots, 1.$
Since any two of these are coprime, we have obtained the order, i.e., $\epsilon_{x,y}$ is indeed of order $(k-1)(k+n-1)$.
\end{proof}

We have established the order of a divisor in one factor of the critical group.
To prove the explicit structure of the critical group, we require both a strict lower and strict upper bound on the rank of the group. 
To do this, we observe that the generators of each factor of the critical group $\eta_{x,y,i}$ can be used to construct a strict lower bound for the rank of the group by showing that any linear combination of these divisors is linearly equivalent to the zero divisor if all divisors are multiplied by $0$ or an integral multiple of $k-1$.
Recall that $\eta_{x,y,i}$ denotes the divisor $\eta_{x,y}$ with a $1$ and $-1$ on adjacent cycles in a fixed orientation of the graph. 
It is worth mentioning that this is not a graph invariant; it is a convention we adopt, which will be something to account for in Section \ref{sec:different_bases}.

\begin{lemma}\label{lemma:eta_basis}
The divisors $\eta_{x,y,i}$ with nonzero entries on adjacent base shapes form a linearly independent set over the group $(\Z / (k-1)\Z)^{n-1}$.
\end{lemma}

\begin{proof}
To show these divisors are linearly independent, we must show that any linear combination $a_1\eta_{x,y,1} + \cdots + a_{n-1}\eta_{x,y,{n-1}}$, for $0 \leq a_1,\dots,a_{n-1} \leq k-2$ on the graph $\mathcal{H}_{k,n}$ is not linearly equivalent to the zero divisor unless all coefficients are zero.

To do this, we necessitate a systematic way of cataloguing our $\eta_{x,y,i}$'s.
In this case, since all the base shapes are identical, we can take all $\eta_{x,y,i}$'s to be adjacent, such that the negative value of one divisor is the positive value of the adjacent divisor. 
See Figure \ref{fig:nonexample} for an example of what can go wrong for arbitrary selection of $\eta_{x,y,i}$.
By this convention, the integers associated to any vertex for some arbitrary linear combination of these divisors is always less than $k-1$ or greater than $-k+1$, since we cannot have linear combinations of two negative or two positive elements. 

As we have seen in Proposition \ref{lemma:order_of_divisors}, $k-1$ is the smallest multiple of any divisor $\eta_{x,y,i}$ where we obtain linear equivalence to the zero divisor. 
In the most general setting, we can think of the linear combination of $\eta_{x,y,i}$'s as one $\eta_{x,y}$ divisor with nonzero entries on every cycle. 
However, the same principle holds for each cycle; there is no firing procedure to obtain linear equivalence to a row of zeroes on each cycle, since the smallest multiple on each cycle was $k-1$.
Thus this linear combination is not equivalent to the zero divisor unless all coefficients are $0$, completing the proof.
\end{proof}

\begin{figure}[ht]
    \centering
    \includegraphics[scale=.25, trim = 3cm 0.5cm 32cm 0.5cm, clip]{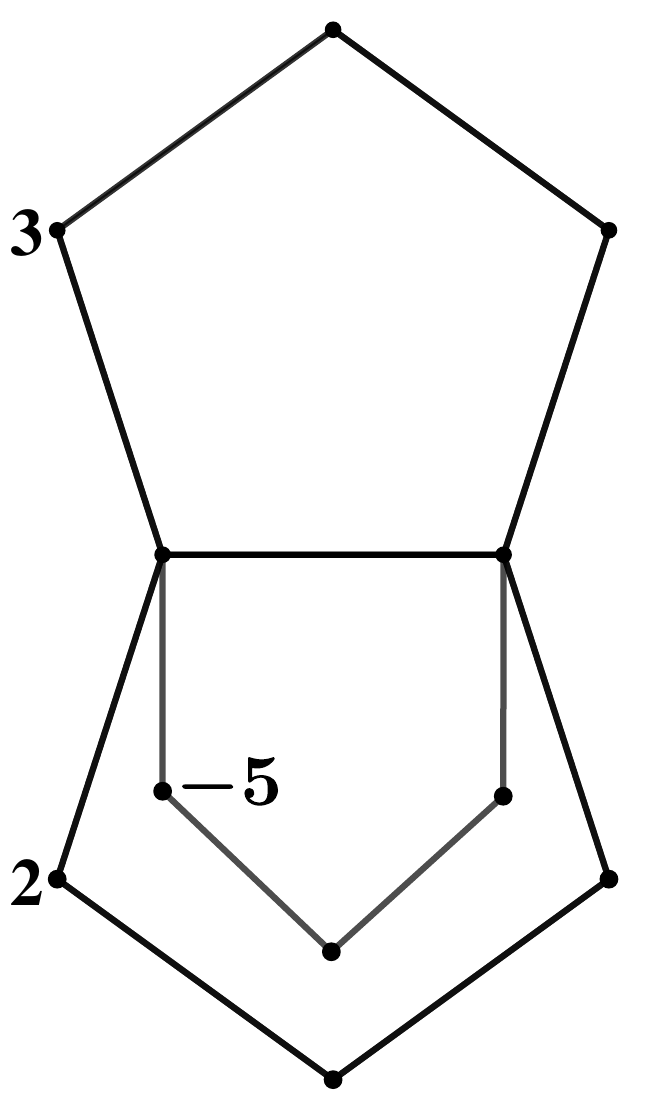}
    \caption{A linear combination of two $\eta_{x,y}$ divisors. This is an example of what could go wrong with linear combinations of arbitrary $\eta_{x,y}$ divisors, since $-5$ is greater in magnitude than $k-1$, so our proof fails and we may no longer have linear independence. We must be careful in general, either selecting adjacent elements for the same base shape, or a generalization for different base shapes.}
    \label{fig:nonexample}
\end{figure}

\begin{corollary}
For $n\geq 3$, the critical group of hinge graphs $K(\mathcal{H}_{k,n})$ is not cyclic.
\end{corollary}

\begin{proof}
 This follows immediately from Lemma \ref{lemma:eta_basis}, since if there are at least $2$ generating elements and consequently $3$ base shapes, $K(\mathcal{H}_{k,n})$ must have rank at least $2$.
\end{proof}

We have established that that any linear combination of the divisors $\eta_{x,y,i}$ is linearly equivalent to the zero divisor if all divisors are multiplied by $0$.
Additionally, since $(k-1)\eta_{x,y,i}$ is linearly equivalent to $0$ for all $i$, any linear combination is also linearly equivalent to $0$.
Because of this, we are able to exhibit a bijection between $\eta_{x,y,i}$ and all of the generating elements of order $k-1$ in the group: $$(1,0,\dots0), (0,1,\dots,0), \dots, (0,0,\dots,(k+n-1)).$$ 
This provides a lower bound for the number of copies of $\Z/(k-1)\Z$; since if we cannot get from any one element of the minimal set to any other via linear equivalence, there must be at least $n-1$ factors in the direct sum. 
Of course, this construction does not provide an upper bound.
Hence, we cannot be sure at this point that this is a minimal generating set.

However, the existence of $\epsilon_{x,y}$, the divisor of order $k^2-1+(n-2)(k-1)$, immediately quells this doubt since a divisor of this order would not exist if there were more copies of $\mathbb{Z}/(k-1)\mathbb{Z}$.
Therefore, we also have an explicit upper bound on the number of factors of the direct sum, and we will have proven the group structure. 
In particular, we have exactly $n-1$ partitions of the group, which identically matches the $n-2$ copies of the smaller factor and the $1$ copy of the large subgroup.

\begin{theorem}\label{thm:critical_group}
The critical group $K(\mathcal{H}_{k,n})$ is isomorphic to 
$$(\Z/(k-1)\Z)^{n-2}\oplus (\Z/(k-1)(k+n-1)\Z).$$
\end{theorem}

\begin{proof}
Considering $\eta_{x,y,i}$ as a minimal generating set, coupled with the order of $\epsilon_{x,y}$, allows us to obtain the exact structure of the critical group for hinge graphs.
We have found an element of order $(k-1)(k+n-1)$, and a minimal generating set of $n-1$ elements of order $k-1$. 
Each of these elements are linearly independent as viewed in the group $(\Z / (k-1)\Z)^{n-1}$, meaning that none of these elements are multiples of others.
This tells us that the rank of the critical group, as defined in \cite{GlassKaplan}, is at least $n-1$, since $k-1$ divides both itself and $(k-1)(k+n-1)$ and by the Fundamental Theorem of Finitely Generated Abelian Groups, this uniquely defines the rank of $K(\mathcal{H}_{k,n})$.
Since we proved in Theorem \ref{thm:order} that the order of $\mathcal{H}_{k,n}$ was $(k-1)^{n-1}(k+n-1)$, the rank has to be exactly $n-1$, where one factor is $(k-1)(k+n-1)$.
Thus, there are $n-2$ small factors and the critical group is $(\Z/(k-1)\Z)^{n-2} \oplus \Z/((k-1)(k+n-1))\Z$.
\end{proof}


\section{Hinge graphs with different base shapes}\label{sec:different_bases}

Using similar techniques as in Section \ref{sec:samebase}, we generalize our results to hinge graphs with different base shapes. 
In this section, we use the notation $\mathcal{H}_{k_1-1,k_2-1,\dots, k_n-1}$ for a hinge graph with different cycles, where again $k_i$ refers to the number of vertices of each base shape.  
This notation differs from the case with identical base shapes since we use $k_i-1$ instead of $k_i$ for its usefulness in the proofs of the results.
The proofs of our results delve into casework involving underlying number theory, for which we provide proof outlines when necessary to guide the reader.

\begin{theorem}\label{thm:general_order}
Consider a hinge graph with different base shapes $\mathcal{H}_{k_1-1,\dots,k_n-1}$, the order of\\ $K(\mathcal{H}_{k_1-1,\dots,k_n-1})$ is
\begin{equation}\label{eq:order}
|K(\mathcal{H}_{k_1-1,\dots,k_n-1})| = a + a/(k_1-1) + \dots + a/(k_n-1),
\end{equation}
where $a:=(k_1-1)\cdots(k_n-1)$.
\end{theorem}

\begin{proof}
Due to Theorem \ref{thm:trees}, the order of the critical group $|K(\mathcal{H}_{k_1-1,\dots,k_n-1})|$ is equal to the number of spanning trees of the graph.
Therefore, it is enough to follow a similar iterative strategy to the one described in the proof of Theorem \ref{thm:order}, to determine the number of spanning trees once a new base shape is added.
This is because the iterative strategy does not rely on the graph being composed of identical base shapes; the only point to consider is that instead of $(k-1)^n$, we now have $(k_1-1),\dots, (k_{n-1}-1)$.
Therefore, showing algebraic equivalence to the iterative strategy is sufficient.

We proceed by induction.
The base cases for one and two base shapes are shown below. 
On the left-hand side is the iterative approach, and the right is the closed form:
\begin{align*}
k_1 &= (k_1-1)+1, \\
k_1 k_2-1 &= k_1 k_2-k_1-k_2+1+(k_1-1)+(k_2-1)=(k_1-1)(k_2-1)+(k_1-1)+(k_2-1).
\end{align*}
For the inductive step $(n \geq 3)$, we can factor a $k_n-1$ out of every term of the closed form except the last as follows, denoting $(k_1-1)\cdots(k_{n-1}-1)$ as $a_{n-1}$ for simplicity:
\begin{align*}
    (k_n-1)(a_{n-1}+a_{n-1}/(k_1-1) + \cdots + a_{n-1}/(k_{n-1}-1)) + ((k_1-1)\cdots(k_{n-1}-1)).
\end{align*}
The first term corresponds to the number $(k_n-1)(S(n-1,k))$ in the iterative approach, which is the number of spanning trees where one edge is removed from the newly added base shape. 
The second term corresponds to the number of spanning trees with the shared edge removed in the iterative approach.
Since these are the only two possibilities, we have found the number of spanning trees, thus completing the proof.
\end{proof}

\begin{remark}
In \cite{GaudetJensenRanganathan} the authors study a family of graphs which they refer to as $\mathbf{s}$-subdivided banana graphs, these are equivalent to thick cycle graphs. 
They determine the order of the critical group of $\mathbf{s}$-subdivided banana graphs for the special case when they let $\mathbf{s}=(s_1,\dots, s_m)$ be a tuple of positive integers such that $\sum_{i=1}^m\frac{\prod_{j=1}^m s_j}{s_i}=p^r$ and $\gcd(s_i,p)=1$ for fixed prime $p$ and integer $r$  (Proposition 13, \cite{GaudetJensenRanganathan}).  
Their proposition follows as a corollary to Theorem \ref{thm:general_order} when interpreting $\mathbf{s}$-subdivided banana graphs as hinge graphs.
\end{remark}

\begin{proposition}\label{diffdelta}
The order of $\delta_{x,y}$, as defined in Section \ref{sec:prelims}, is
$$|\delta_{x,y}|=b+b/(k_1-1)+b/(k_2-1)+\cdots+b/(k_n-1),$$
where $b:= \lcm(k_1-1, \dots , k_n-1)$. 
\end{proposition}

\begin{proof}
As in the proof of Proposition \ref{lemma:order_of_divisors} for the order of $\epsilon_{x,y}$, firing a multiple $m$ of $\mathbf{r}$ is equivalent to repeating the chip-firing process $m$ times, and therefore directly corresponds to multiplying the values associated with each divisor by $m$.  
Consequently, given a specific number of vertices, firing adjacent vertices by a multiple of a consecutive number of times still produces a zero on each vertex.
This is because each vertex is connected to two others, and hence when we subtract the same multiple repeatedly we are eliminating that deficit.
Thus, the question of determining $|\delta_{x,y}|$ first involves finding the least common multiple of each $k_i-1$, such that a multiple of the consecutive chip-firing process can be applied to each cycle.
Refer to Figure \ref{figb} for an example of this process applied to the graph $\mathcal{H}_{2,3,4}$. 
Notice that in the figure, one of the shared vertices is never fired.
This is valid as the graph Laplacian $L$ has kernel of dimension $1$, and therefore we can choose a vertex not to fire.
This also directly corresponds to the use of $k_i-1$, as we are excluding one vertex from each cycle when applying the consecutive process a multiple of times.

\begin{figure}[ht]
    \centering
    \includegraphics[scale=.25, trim = 3cm 2cm 31.5cm 2cm, clip]{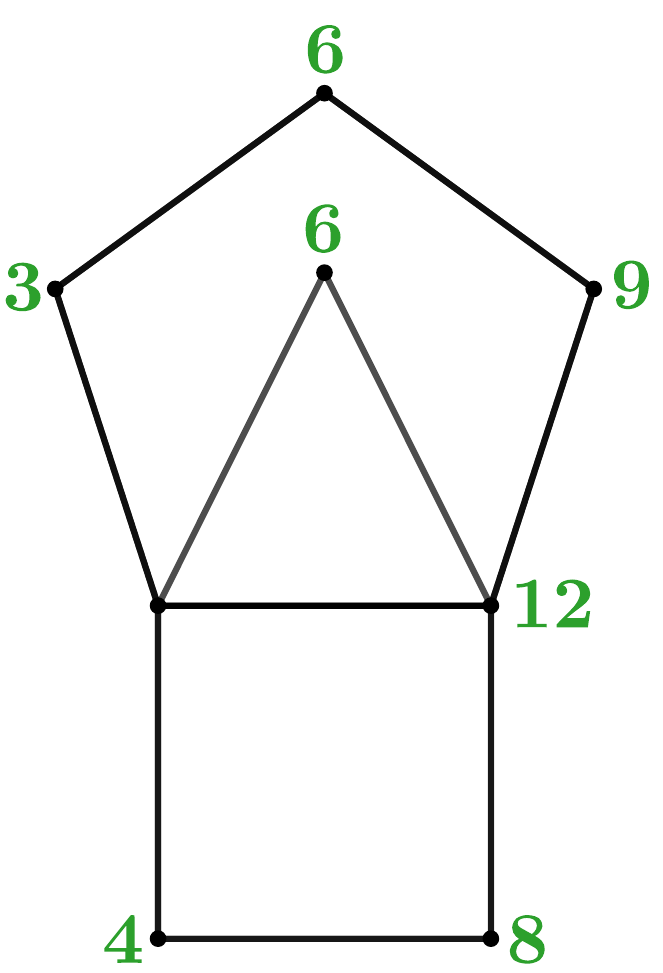}
    \caption{In green are the number of times each vertex is fired, and we can count the ones adjacent to the shared vertex to obtain $|\delta_{x,y}|=25$.}
    \label{figb}
\end{figure}

It remains to be shown that that this is the smallest multiple of $\delta_{x,y}$ linearly equivalent to the zero divisor.
First, by Lemma \ref{lemma:bgcd}, we see that the number of times we fire the vertices on each cycle are coprime.
However, this is not sufficient as established above, since we must also check that adding and subtracting multiples of $\mathbf{1}$ does not affect this.
Hence, we can apply Lemma \ref{thm:invariantgcd} by selecting two base shape whose smallest multiples of firings, i.e., the number of firings on the vertex adjacent to the shared edge, are coprime.
This is guaranteed to occur by Lemma \ref{lemma:bgcd}.
From there, Lemma \ref{thm:invariantgcd} tells us that there must exist two vertices whose number of firings differ by 1, and hence we have found that the number of firings are coprime invariant under addition of $\mathbf{1}$ and can conclude that we have found the order of $\delta_{x,y}$.

Finally we will compute this order.
By considering the number of times we fire all the vertices adjacent to the vertex that was not fired, we can explicitly find $|\delta_{x,y}|$.
From each cycle, the number of times fired to the unfired vertex is $b/(k_i-1)$, where $b:=\lcm(k_1-1,\dots, k_n-1)$, because we are applying $b/(k_i-1)$ multiplied by the consecutive chip-firing process on each cycle.
Then, from the vertex on the shared edge, we will simply have $b$, since this is required for the chip-firing process on each cycle.
Note that the vertex not fired must correspond to a multiple of the $-1$ in $\delta_{x,y}$, since it receives chips from every cycle; likewise the other vertex must correspond to the $+1$.

Hence, $|\delta_{x,y}|$ is $b+b/(k_1-1)+ \dots + b/(k_n-1)$ as required.
\end{proof}

We now turn our attention to the divisor $\epsilon_{x,y}$. 
For different base shapes, however, the particular base shape we choose for $\epsilon_{x,y}$ matters, so we will refer to the divisor as $\epsilon_{x,y,i}$, where the $i$ denotes the base shape on which the vertex corresponding to $1$ resides.
Additionally, there are conditions associated with the order of $\epsilon_{x,y,i}$, namely that we can extract a multiple of $k_i-1$ from the order of the critical group.

\begin{theorem}\label{thm:diffepsilon}
The divisor $\epsilon_{x,y,i}$ on a chosen shape with $k_i$ vertices has order $$(k_i-1)|\delta_{x,y}|,$$ provided there are at least two base shapes with $k_i$ vertices or there exists another base shape with  $t(k_i-1)+1$ vertices for some $t \in \N$.
\end{theorem}

\begin{proof}
The proof follows a similar approach as the proof of the order of $\epsilon_{x,y,i}$ in Proposition \ref{lemma:order_of_divisors}, except we must additionally keep track of the number of vertices of each base cycle.
By applying Lemma \ref{lemma:bgcd}, we have that the number of times we fire each vertex adjacent to the shared vertex not fired is coprime.

We want to show that $(k_i-1)\epsilon_{x,y,i}$ is linearly equivalent to $\delta_{x,y}$ and that combining chip-firing operations show the orders are multiplicative.
We consider how this affects the number of times we fire each cycle.

Suppose there exists a base shape $C$ with $|V(C)|=k_i$ such that there exists no other base shape with number of vertices equal to $t(k_i-1)+1$. 
In other words, there are no multiples of $k_i-1$ in $\{(k_1-1), \dots, (k_n-1)\}$.
By Equation (\ref{eq:order}), there will not be a common factor of $k_i-1$ in $\frac{|\mathcal{H}_{\{k_1-1, \dots, k_n-1\}}|}{|\delta_{x,y}|}$, which reduces to $a/b$.
Hence, $\epsilon_{x,y,i}$ will not have order $(k_i-1)|\delta_{x,y}|$ in this case.
Since they do not satisfy the conclusions of the theorem, we will not consider these to be $\epsilon_{x,y,i}$'s in all following proofs.

On the contrary, suppose that there does exist such a cycle. Then $k_i-1$ divides $\lcm(k_1-1,\dots, k_{i-1}-1, k_{i+1}-1,\dots, k_n-1)$.
Thus, we can ignore this cycle and use the consequence in Lemma \ref{lemma:bgcd} on the $n-1$ other base shapes to prove that they will all be fired a coprime number of times.

To make sure this is invariant under addition of the vector $\mathbf{1}$, we once again apply Theorem $\ref{thm:invariantgcd}$.
Since we have shown each will be fired a coprime number of times, the conditions for this theorem are met and there are two vertices for which the number of times fired differs by $1$.
Thus, it is coprime irrespective of addition of the kernel, and we can be certain that we have found the order of $\epsilon_{x,y,i}$.
\end{proof}

\begin{remark}\label{rmk:lcmep}
If there exist multiple $\epsilon_{x,y,i}$ for which $|\epsilon_{x,y,i}|/|\delta_{x,y}|$ are coprime, then the order of the large factor of the critical group will be at least $\lcm(|\epsilon_{x,y,i}|/|\delta_{x,y}|)\cdot |\delta_{x,y}| = \lcm(|\epsilon_{x,y,i}|)$.
This is an important consequence that we will invoke to prove the order of the critical group.
\end{remark}
 
In the following theorem we prove the orders of the minimal generating elements of the critical groups of the hinge graphs.  
Recall that for the hinge graphs with the same base shape, the number of factors in the direct sum of $K(\mathcal{H}_{k,n})$ and their size was determined entirely by the minimal generating set of divisors $\eta_{x,y,i}$.
Furthermore, all of these divisors were of order $k-1$. 
In what follows, we show that such a set of divisors exists when the conditions for the theorem above hold.

In the case where all base shapes have identical numbers of vertices, we could take the minimal generating set to be $\eta_{x,y,i}$, where each $i$ denotes a divisor on adjacent pairs of base shapes. 
A consequence of taking divisors on pairs of cycles is that we require two generating divisors to construct two factors in the structure of $K(\mathcal{H}_{k,n})$, and therefore we must look at adjacent subsets of size three of $\{k_i-1\}$ for $i=1,\dots,n$.
This result is not generalizable, since the order in which each shape appears directly affects the calculated order of elements in the minimal generating set.
For instance, by taking adjacent subsets of size three of the hinge graph $\mathcal{H}_{\{4,4,6,6,12,12\}}$ we obtain a different set of orders of generating elements than that of $\mathcal{H}_{\{4,4,12,12,6,6\}}$, despite the graphs being isomorphic.
The following theorem generalizes the principle of treating collections of size three in cases where the adjacent construction is no longer possible.

\begin{theorem}\label{thm:diffcyclebasis}
A linearly independent set of divisors for the critical group can be constructed as one of the following:
\begin{enumerate}
  \setlength\itemsep{.5em}
\item[(i)] For any collection $s$ of identical base shapes, each with $k_i$ edges, such that $|s| \geq 3$, there exist $|s|-2$ linearly independent divisors of order $k_i-1$ in the critical group.

\item[(ii)] For any pair of base shapes where each base shape has $t_i$ edges, with the property that the greatest common divisor of $t_i-1$ with all other $(k_j-1) \neq (t_i-1)$ is not $1$, there exists a divisor of order $c:=\gcd(k_j-1, t_i-1))$. 
(Note that if there are three or more identical base shapes, each with $t_i$ edges, only one pair with order $t_i-1$ is counted.)

\item[(iii)] For each $k_i-1$ that appears only once in $\{k_1-1,\dots, k_n-1\}$ and satisfies \[\gcd((k_1-1)\cdots(k_{i-1}-1)(k_{i+1}-1)\cdots(k_n-1), \hspace{0.15cm} (k_i-1)) \neq 1,\] there exist divisors of the critical group of order $\gcd(c, \hspace{0.15cm} k_i-1)$, where $c$ is the $\gcd$ of $(k_i-1)$ with 2-tuples of $\{(k_1-1),\dots,(k_{i-1}-1),(k_{i+1}-1),\dots,(k_n-1)\}$. 

\end{enumerate}
Furthermore, in items (ii) and (iii), the critical group obtains a small factor group of at least $\Z/\lcm([c_j])\Z$, where $[c_j]$ denotes the collection of all such $c$ for a particular pair with $t_i$ edges or unique base shape with $k_i$ edges. 
\end{theorem}

Before we begin the proof, we provide an example of each item of Theorem \ref{thm:diffcyclebasis}.
In all following examples, the sequences of numbers given will be $\{k_1-1, \dots, k_n-1\}$.

\begin{example} \hfill
\begin{enumerate}
  \setlength\topsep{.5em}
  \setlength\itemsep{.5em}
    \item[(i)] Take the collection $\{2,2,2,3,5\}$, which corresponds to the hinge graph with three triangles, a square, and a hexagon.
By looking at the collection $s$ of triangles, there must be a factor of $\Z/2\Z$ in the critical group, identical to the outcome of Lemma \ref{lemma:eta_basis}.

\item[(ii)] Take the collection $\{2,2,4,4,5\}$, which corresponds to the hinge graph with two triangles, two pentagons, and a hexagon.
Although this collection does not contain three or more identical copies of the same base shape, we have a pair $t_1=3$ of triangles and a pair $t_2=5$ of pentagons. Since $\gcd(t_1-1, k_3-1) = 2$ and $\gcd(k_1-1, t_2-1)=2$, we obtain two copies of $\Z/2\Z$.
More concretely, this comes from looking at $3$-tuples of the original collection $\{2,2,4\}$ and $\{2,4,4\}$, both of which give us two linearly independent divisors (that are linearly independent from each other by Lemma \ref{lemma:eta_basis}), one which contributes to a unique small factor group and the other which is in the largest factor group.

\item[(iii)]  Take the collection $\{2,2,3,3,12\}$, which corresponds to the hinge graph with two triangles, two squares, and a tridecagon. By looking at the greatest common divisors of pairs of elements in $\{2,2,3,3\}$ with 12, we obtain a copy of $\Z/3\Z$ from $\{3,3,12\}$ and a copy of $\Z/2\Z$ from $\{2,2,12\}$.
Since $\Z/2\Z \times \Z/3\Z \cong \Z/6\Z$, the critical group contains a factor of $\Z/6\Z$. 
\end{enumerate}

\end{example}

\begin{proof}[Proof of Theorem \ref{thm:diffcyclebasis}]

We split the proof into two parts: 
\begin{enumerate}
  \setlength\topsep{.5em}
  \setlength\itemsep{.5em}
  \item[(1)] First, we outline an iterative procedure, starting with the conclusion of Lemma \ref{lemma:eta_basis} to generate small factors of the critical group in the case when not all base shapes are identical.
    \item[(2)] Then, each item of Theorem \ref{thm:diffcyclebasis} is shown to be a consequence of this iterative procedure.
    This is split into three cases when: 
        \begin{enumerate}
        \setlength\itemsep{.5em}
        \item there is a collection of $\geq 3$ identical base shapes,
        \item there is a pair of base shapes with $t_i$ edges for which the greatest common divisor of $t_i-1$ with the other $k_i-1$'s in the collection is not 1.
        This case also splits into three additional cases (detailed in the proof below), and
        \item there is a unique $i$th base shape for which the greatest common divisor of $k_i-1$ with $2$-tuples of $k_j-1$ for $k_i-1 \neq k_j-1$ is not 1.
        \end{enumerate}
\end{enumerate}
\hfill\\
\noindent \textbf{Part (1):} We begin with two identical base shapes that are a subset of the larger hinge graph, and we know $\eta_{x,y}$ is an element of order $k-1$.
To determine the next base shape with the smallest number of vertices such that the order remains $k-1$, we add one vertex in between the nonzero vertex of $\eta_{x,y}$ and the shared edge.
Consequently, the positive vertex of $\eta_{x,y}$ must be fired $k-2$ additional times, and the newly added vertex must be fired $k-2$ times.
This is because with two copies of the same base shape, applying chip-firing operations to ``push" $k-1$ onto the shared edges always results in a $k-2$ and $1$ on the shared vertices, as seen in the two leftmost diagrams of Figure \ref{fig:pic2}. 
Therefore, we fire each vertex in that path $k-2$ more times to ``push" the $k-2$ assigned to the vertex one additional vertex further. 
This procedure is reminiscent of the strategy depicted in Figure $\ref{fig:fireprocess}$, involving utilization of Theorem \ref{thm:obs1}, where the difference is that we multiply the consecutive chip-firing process by $k-2$, and this only applies to the left side of the base shape.

\begin{figure}[ht]
    \centering
    \includegraphics[scale=.30, trim = 6cm 4.5cm 7cm 4cm, clip]{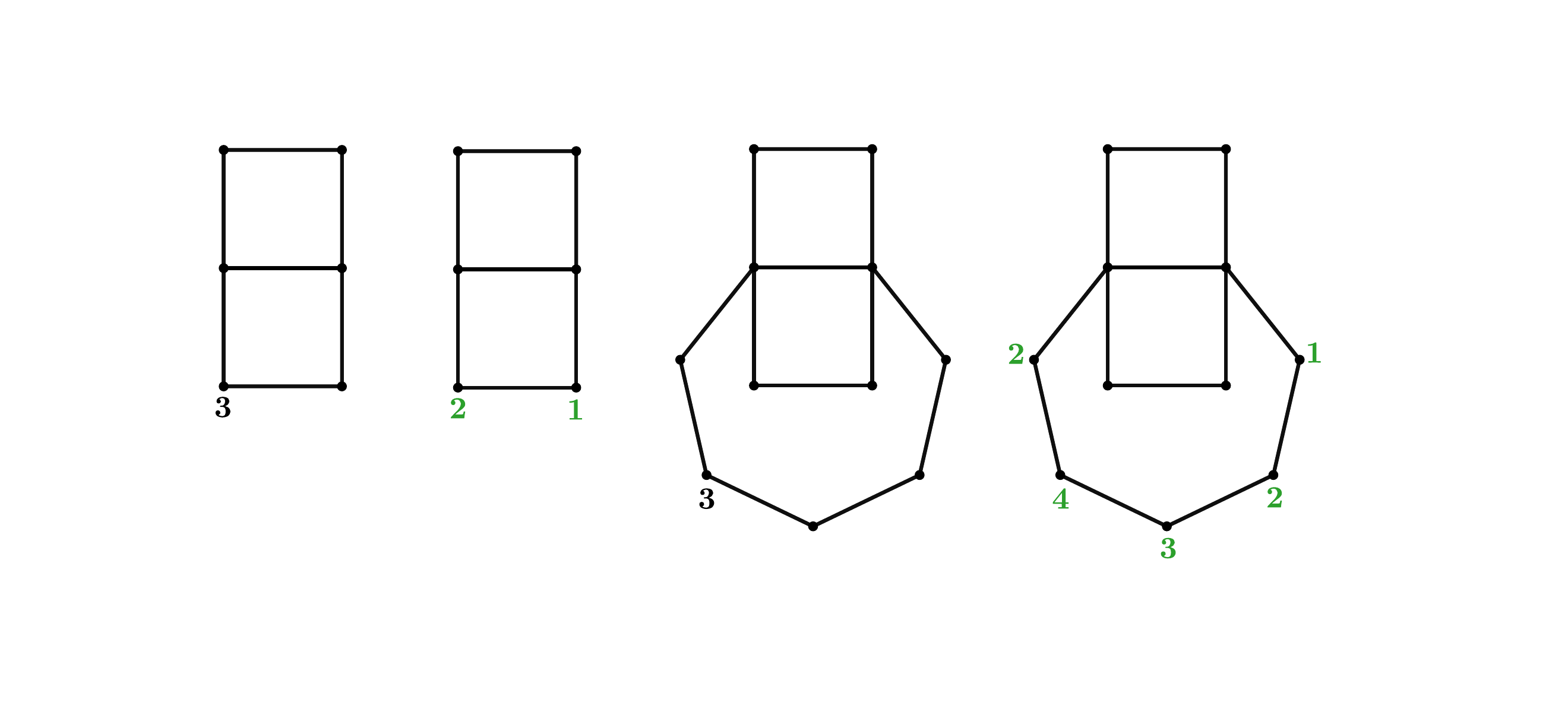}
    \caption{The process of finding the next smallest shape with a divisor of the same order. A vertex is added on the left path, and consequently $k-2$ vertices are added to the right path. The firing procedure can be seen on the right.}
    \label{fig:pic2}
\end{figure}

Now that we have increased the number of times we fire the vertex of $\eta_{x,y}$ by $k-2$, we must still be able to ``push" the $1$ through the base shape onto the other shared vertex.
We must apply the consecutive chip-firing process, where we fire each successive vertex $k-2, k-3,\dots k-k$ times.
Firing the vertex of $\eta_{x,y}$ $k-2$ times, and then firing the vertex between $\eta_{x,y}$ and the shared vertex $k-2$ times results in no net change to the integer associated with this vertex, and this procedure still results in the same configuration. 
However, $k-2$ additional vertices must also be added to the rest of the base shape (e.g., the right path in the rightmost diagram of Figure \ref{fig:pic2}) such that the consecutive chip-firing process can be applied and the $1$ can migrate to the shared edge. 

Note that we have added one vertex on one side of the base shape and consequently, added $k-2$ vertices on the other side of the base shape. 
Thus, the total number of newly introduced vertices is $k-1$.
This is the least number of vertices that can be introduced, since we cannot add fractional numbers of vertices. 
Repeating this process for any number of base shapes an arbitrary amount of times, we see that a base shape with number of vertices $k_i$ is not necessary for a factor of order $k_i-1$. 
(As an example, the hinge graph consisting of two pentagons and a heptagon, $\mathcal{H}_{4,4,6}$, has a factor of order $2$ by applying this procedure to all three base shapes, despite not having any triangle base shapes.)
As in the conclusion of Lemma \ref{lemma:eta_basis}, we require three base shapes in order to have two linearly independent divisors, one of which goes into the large factor of the critical group and the other creates a unique small factor; thus, it is sufficient to look at subsets $\{k_1-1, \dots, k_n-1\}$ of size three.
Then, taking the greatest common divisor of the chosen elements provides a split factor in the critical group of that order. 

\noindent \textbf{Part (2):} Next, we consider the possible outcomes of applying the procedure presented in Part (1).
\begin{enumerate}
\setlength\itemsep{.5em}
\item[(i)] If there exists a collection $s$ containing three or more identical base shapes with $k_i$ edges, the same procedure applies as in Lemma \ref{lemma:eta_basis} and we obtain $|s|-2$ copies of $\Z/(k_i-1)\Z$.

\item[(ii)] For any pair of identical base shapes with number of edges $t_i$ whose $\gcd$ with some $k_j-1$ in the rest of the collection is equal to $c \neq 1$, i.e., $\gcd(k_j-1, t_i-1) = c \neq 1$, apply the procedure from Part (1).
The procedure is applied recursively from a cycle with $c+1$ edges (which may not be present in the hinge graph) to each base shape in the triple (the pair with the added shape) as necessary, and hence we obtain two linearly independent divisors of order $c$. 
Fixing $i$, recall that $[c_j]$ denotes the set of all $c_j$ obtained by taking $\gcd(k_j-1, t_i-1) = c_j \neq 1$, for some $k_j-1 \neq t_i-1$ in the collection $\{k_1-1, \dots k_n-1\}$.
We claim that the divisors of order $c_j$ are linearly dependent for fixed $i$. 
Since our $3$-tuple of $k_i-1$'s includes a pair of identical base shapes, we can guarantee (via adding our divisors together) that there exists a divisor with a $1$ and $-1$ on the pair with number of edges $t_i$ (see Figure \ref{fig:new} for an example).
Thus, all of our divisors from this procedure in this case can be thought of as having a $1$ and $-1$ on our pair of base shapes, except at different vertices.
There are three possibilities:

\begin{enumerate}
  \setlength\itemsep{.5em}
  \item If a divisor of order $c_j$ is coprime to all others constructed from the procedure in Part (1), linear dependence is irrelevant because the Chinese Remainder Theorem guarantees that the order of the group will be at least the product.
    
    \item If the order of one divisor is a multiple of another, then we can use the consecutive chip-firing process on each one (firing on one, borrowing from the other) to show they are equivalent.
    This is a consequence of the construction of the divisors from Part (1).
    Importantly, each time we consider the next smallest base shape with a divisor of order $k_i-1$, we add another vertex in between the nonzero vertex and the shared edge, while also increasing the multiple of $k_i-1$ by 1.
    Thus, taking a multiple of the divisor, we can use the consecutive chip-firing process, much like in the proof of the order of $\eta_{x,y}$ in Proposition \ref{lemma:order_of_divisors}, to get a $1$ on the vertex which now corresponds to the divisor of smaller order.
    The analogous procedure on the opposite base shape in the pair, but borrowing instead of firing, produces a $1$ and $-1$ on the vertices desired, whereas the nonzero element on the shared vertex is cancelled out by this procedure (see Figure \ref{fig:today} for an example).
    Thus, if a $c_i$ is a multiple of another $c_j$, then they are linearly dependent.

    \begin{figure}[ht]
    \centering
    \includegraphics[scale=.30, trim = 2cm 3cm 8.4cm 2.5cm, clip]{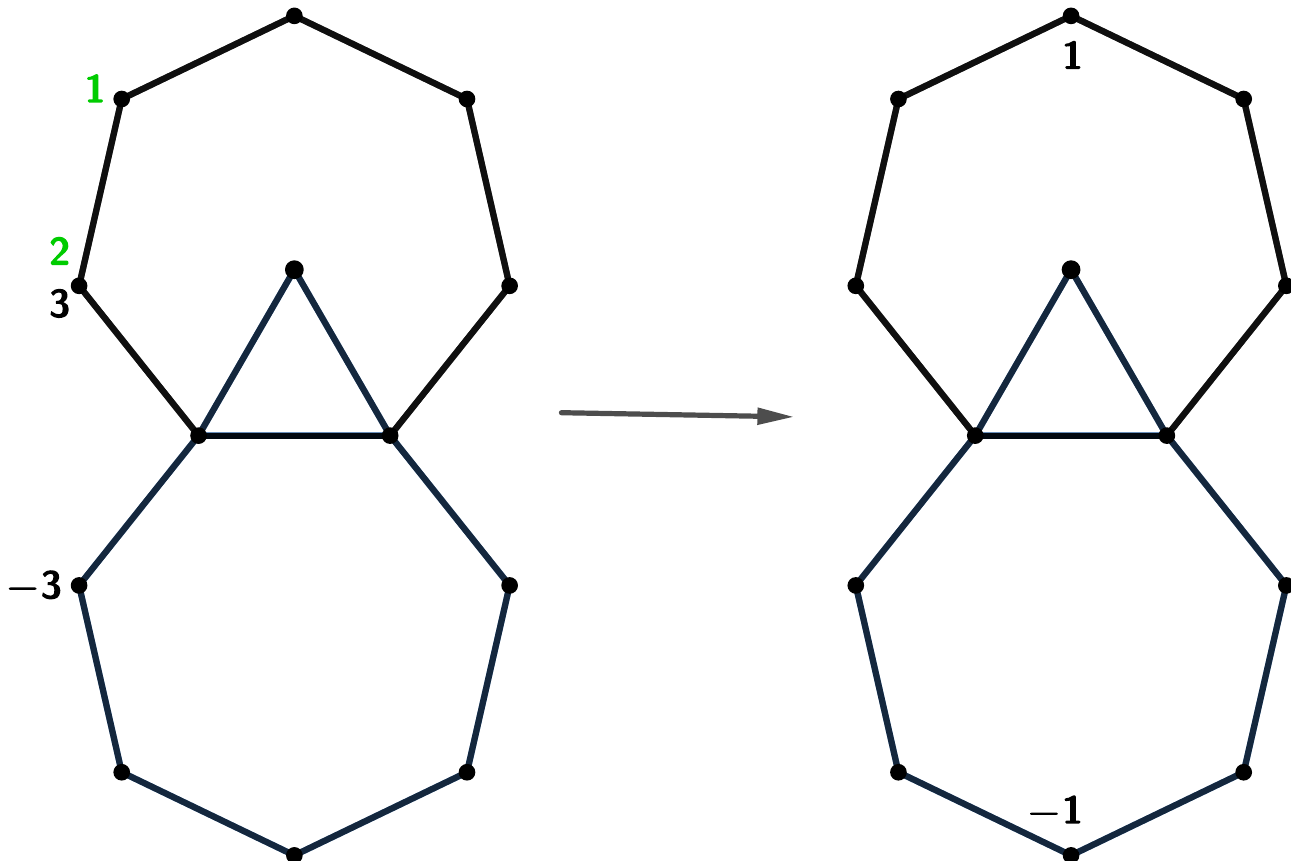}
  \caption{A multiple of an order $6$ divisor constructed from the procedure in Part (1) is linearly equivalent to a divisor of order $2$ on the same pair of base shapes. The black numbers are the integers associated to the vertices, and the green numbers are the number of times fired.}
    \label{fig:today}
\end{figure}
    
    \item If a subcollection of the orders $[c_i]$ of the divisors are not coprime and also not multiples of each other, the pair of base shapes with $t_i$ edges will contain a divisor $d$ whose order is the $\gcd$ of the $c_i$ in this subcollection.
    By subcase (b), since multiples of all $c_i$ should be linearly equivalent to $d$, the divisors with orders in this subcollection are linearly dependent, and we can conclude that the small factor group must have order at least the $\lcm$ of the original orders. 
\end{enumerate}

\noindent Consequently, the order of the small factor must be at minimum the least common multiple of the orders of the divisors $c_j$, so we obtain a copy of at least $\Z/\lcm([c_j])\Z$ as a small factor of the critical group.

\begin{figure}[ht]
    \centering
    \includegraphics[scale=.30, trim = .8cm 2.5cm 29.9cm 1.1cm, clip]{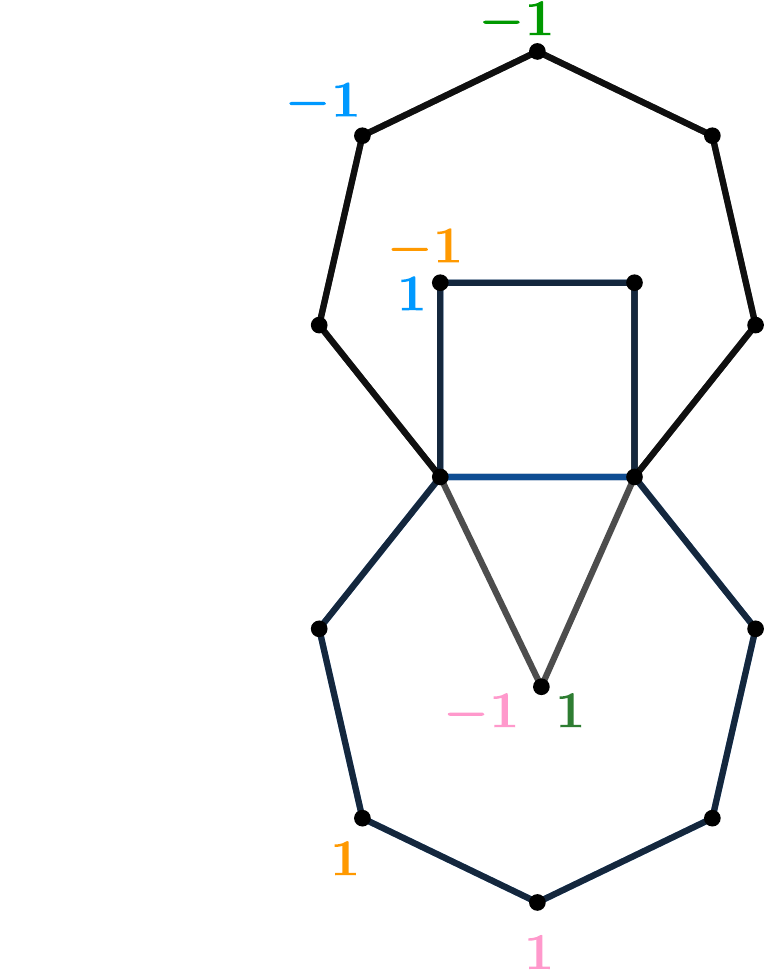}
    \caption{An example of the divisors constructed via the procedure outlined in Part (1) on $\mathcal{H}_{2,3,6,6}$. Adding the green and red divisors creates a divisor of order $2$ on the pair with $7$ edges, and adding the blue and orange divisors creates a divisor of order $3$ on the same pair. In this case, $2$ and $3$ are coprime, so the critical group has a factor of at least $\Z/6\Z$.}
    \label{fig:new}
\end{figure}

\item[(iii)] Lastly, if there exists a unique $k_i-1$ whose greatest common divisors with pairs of other elements is $c_i \neq 1$, then applying the procedure from Part (1) to a base shape with $c+1$ edges gives two linearly independent divisors of order $c$, implying the existence of an additional small factor in the critical group.

Consider the greatest common divisors $c_i$ of $k_i-1$ with pairs of all other elements in $\{k_1-1, \dots, k_n-1\}$.
The issue is that now, the divisors of order $c_i$ corresponding to a particular $k_i-1$ are not linearly dependent.

Nevertheless, we still claim that any $c_i$ which are factors of other ones are not considered in the critical group.
To see this, consider the collection $\{k_i-1, \dots, k_{i-1}-1, k_{i+1}-1, \dots, k_n-1\}$, i.e., remove the $i$th base shape. 
If there exist two $2$-tuples for which one $c_i$ was a factor of the other $c_j$, then a factor of the critical group $\Z/c_i \Z$ already exists.
By the procedure of Part (1), when adding back in the $i$th base shape there must exist a linearly independent divisor of order $c_j$. 
Hence, if we obtained a second copy of $\Z/c_i \Z$ by adding back in the $i$th base shape, this subcollection of size $4$, which includes a $2$-tuple, our original $k_i-1$, and one other element of $\{k_i-1, \dots, k_n-1\}$, will contribute to $3$ linearly independent small factor groups. 

However, in the case where all base shapes are identical and we can construct a maximal set of $\eta_{x,y,i}$'s, by Theorem \ref{thm:critical_group} there are $n-2$ small factors in the critical group, and this must hold for the generalization in Part (1) of this proof.
Indeed, if we consider the generalized $\eta_{x,y,i}$'s on these $4$ base shapes, there would be $4$ of them, where one contributes to the large factor group. 
This collection cannot possibly be a maximally generating set, as taking linear combinations forces two divisors to be on the same pair of base shapes, where one is a multiple of the other, and the consequence of Part (1) used in item (ii) implies they are linearly dependent. 

Thus, only the largest factors are relevant and we obtain a small factor of the critical group of at least $\Z/\lcm([c_i])\Z$.
Of course, we cannot know for certain whether this is the least common multiple, or if is a product of $\Z/c_i\Z$. 
However, this guarantees a lower bound which we will later show is exact.
\end{enumerate}

\noindent As these are all of the possibilities outlined in the theorem statement, this completes the proof.
\end{proof}

With the following well-known lemma whose proof is omitted, we have all the tools to state and prove the structure of the critical group.

\begin{lemma}\label{thm:gcd}
Let $\{a_1, \dots, a_n\}$ be a set of positive integers. 
Then $$\lcm(a_1,\dots,a_n)=(a_1\cdots a_n)/(\gcd((n-1)-\operatorname{tuples})).$$
\end{lemma}

\begin{theorem}\label{thm:critical}
Assume the notation from Theorem \ref{thm:diffcyclebasis}.
Let $\alpha_p$ denote the $k_i-1$ for each collection $s_p$ in the first bullet of Theorem \ref{thm:diffcyclebasis}. 
Let $\beta_q$ be the least common multiple of the $c_i$ for every pair with $t_q$ edges in the second item, and let $\gamma_r$ denote the least common multiple of the $c_i$'s for every unique $k_r-1$ in the third item.

Then the critical group for the hinge graph $\mathcal{H}_{k_1-1, \dots, k_n-1}$ is isomorphic to
\[\bigoplus_p \Z/\alpha_p \Z \bigoplus_q \Z/\beta_q\Z \bigoplus_r \Z/\gamma_r\Z \oplus \Z/(\lcm(|\epsilon_{x,y,i}|))\Z.\]
\end{theorem}

\begin{proof}
Using the notation as in Theorem \ref{thm:general_order} and Proposition \ref{diffdelta},  observe that the quotient \[|\mathcal{H}_{k_1-1,\dots,k_n-1}|/|\delta_{x,y}|=a/b = (k_1-1)\cdots(k_n-1)/(\lcm((k_1-1),\dots, (k_n-1))).\]
By Lemma \ref{thm:gcd}, this is equivalent to $\gcd((n-1)-\operatorname{tuples})$.
Thus, it suffices to show that the contribution from $\lcm(|\epsilon_{x,y,i}|)/|\delta_{x,y}|$ plus the three bullet points outlined in Theorem \ref{thm:diffcyclebasis} are all possible factors in $\gcd((n-1)-\operatorname{tuples})$.

First, if there exist $k_i-1, k_j-1$ for $i\neq j$ whose greatest common divisor is $c \neq 1$, there exists a divisor $\epsilon_{x,y,i}$ with order $c |\delta_{x,y}|$ in the critical group by Theorem \ref{thm:diffepsilon}.
In the greatest common divisor of $n-1-\operatorname{tuples}$, there must be a factor of exactly $c$ and no larger, since one element in $\gcd((n-1)-\operatorname{tuples})$ is missing $k_i-1$, and another is missing $k_j-1$.

Second, for any collection $s$ of identical base shapes with $k_i$ edges on the hinge graph, there will be $|s|-1$ factors of $k_i-1$ in $\gcd((n-1)-\operatorname{tuples})$ since at most one can be removed.
This corresponds to one factor in $|\epsilon_{x,y,i}|$ from above, and $|s|-2$ copies in the small factors of the critical group as established in Theorem \ref{thm:diffcyclebasis}.

Third, if there exists a pair of identical base shapes with $t_i$ edges, there is a factor of $t_i-1$ in $\gcd((n-1)-\operatorname{tuples})$ from the first point in the proof, but this is not necessarily the only factor that increases.
If $t_i-1$ is not coprime to all other $k_j-1 \neq t_i-1$, then there is another factor of $c = \lcm([c_i]) = \lcm(\gcd(t_i-1, k_j-1))$, where each $c_i$ comes from the $k_j-1$ that was removed from $\gcd((n-1)-\operatorname{tuples})$.
When we take the $\gcd((n-1)-\operatorname{tuples})$, all extra factors will be removed, except for the least common multiple.
Hence, the least common multiple is taken rather than the product.

Fourth, if there exists a base shape with $k_i$ edges, where $k_i-1$ is unique but whose greatest common divisor with two other elements is $c_i \neq 1$, we get two factors of the greatest common divisor $c_i$ in  $\gcd((n-1)-\operatorname{tuples})$ since we only remove one of the three elements.
One of these factors is in $\epsilon_{x,y,i}$ by Theorem \ref{thm:diffepsilon}, and the other is a $c_i$ contributing to the small factor group due to Theorem \ref{thm:diffcyclebasis} in a similar manner to the third point of this proof.

Finally, if there exists a base shape with $k_i$ edges such that $k_i-1$ is coprime to all other $k_j-1$ for $i \neq j$, then this factor does not appear in the greatest common divisor of $(n-1)-\operatorname{tuples}$.
This corresponds to the fact that it is included in $|\delta_{x,y}|$ by Proposition \ref{diffdelta}, and hence has already been quotiented out.

So we have run through every possible case outlined in Theorem \ref{thm:diffcyclebasis} as well as Theorem \ref{thm:diffepsilon}.
We know that this constitutes all factors in the greatest common divisor because if any number is added to the collection that we are taking $\gcd((n-1)-\operatorname{tuples})$ of, then it must satisfy one of the following:
\begin{itemize}
    \item it is coprime to all others,
    \item it is unique and a multiple of other numbers,
    \item it exists in a pair which is coprime from all others, 
    \item it exists in a pair which is not coprime to all others, or
    \item it exists in a collection of 3 or more identical numbers. (Note that the question of whether this is coprime to others is already accounted for by looking at pairs.)
\end{itemize}
But these are exactly the cases we have run through.
By combining Remark \ref{rmk:lcmep}, which gives us an upper bound for the size of the critical group, and Theorem \ref{thm:diffcyclebasis}, which gives us the number of small factors of the critical group, we obtain  

\[\bigoplus_p \Z/\alpha_p \Z \bigoplus_q \Z/\beta_q\Z \bigoplus_r \Z/\gamma_r\Z \oplus \Z/(\lcm(|\epsilon_{x,y,i}|))\Z.\]

as the critical group.
This finishes the proof.
\end{proof}

\begin{example}
As an example, consider the collection $\{k_1-1, \dots, k_n-1 \} = \{4,4,4,6,6,24,7\}$.
In this case, \[\gcd((n-1)-\operatorname{tuples}) = 4*4*4*6*6=2304.\]
From the first part of the proof of Theorem \ref{thm:critical}, we obtain a factor of $\lcm(4,6)=12$.
From the second part, we obtain a factor of $4$ due to the subcollection $\{4,4,4\}$.
From the third part, we get two copies of $2= \gcd(4,4,6)=\gcd(4,6,6)$, one from each pair (making sure not to overcount by considering $24$ as it is a unique element).
From the fourth part, we obtain a factor of $\lcm(\gcd(4,4,24), \gcd(6,6,24))=12$.
Lastly, from the fifth part, we obtain no factors, since $7$ is coprime to all other elements and does not appear in the $\gcd$.
And we see that $12*4*4*12=2304$, as expected.

By Proposition \ref{diffdelta}, the order of $\delta_{x,y}$ is $168+\frac{168}{4}+\frac{168}{4}+\frac{168}{4}+\frac{168}{6}+\frac{168}{6}+\frac{168}{24}+\frac{168}{7} = 381$, so $\lcm(\epsilon_{x,y,i})=381*12=4572.$

Thus, the critical group \[\mathcal{H}_{k_1-1,\dots k_n-1} \cong \Z/2\Z \oplus \Z/2\Z \oplus \Z/4\Z \oplus \Z/12\Z \oplus \Z/4572\Z.\]
\end{example}

An alternative route to obtain the critical group of hinge graphs with different base cycles is to invoke Theorem 2 in \cite{CoriRossin}, which states that the critical group of a planar graph is isomorphic to the critical group of its dual.
Hinge graphs are the duals of thick cycle graphs, thus the critical group of a hinge graph is isomorphic to the critical group of a thick cycle graph, whose complete structure is given in Theorem 2.29 in \cite{Alfaro} and Theorem 1 in \cite{MSRI-UP}.


\section{Further Directions}\label{sec:conclusion}
We conclude this paper by providing some questions and problems that are worth investigating and might be of interest to others. 
    \vspace{0.25cm}

\begin{enumerate}
    \item Suppose the hinge graph is treated as an operation on graphs more generally. How is the critical group of an arbitrary graph affected by the hinge operation? 
    \vspace{0.25cm}

    \item How is the critical group of a graph affected if we attach a base shape to other specified edges?
    \vspace{0.25cm}

    \item Describe the critical group structure of the cone over a hinge graph.
    \vspace{0.25cm}

    \item Can the number of arithmetical structures, as defined in \cite{GlassKaplan}, on hinge graphs be enumerated? 
    \vspace{0.25cm}

    \item In \cite{KeyesReiter}, Keyes and Reiter derive an upper bound for the number of arithmetical structures of connected undirected graphs on n vertices with no loops, which only depends on the number of vertices and edges. 
    When considering arithmetical structures of hinge graphs, can we refine and compare their upper bounds?
    \vspace{0.25cm}

    \item What can be said about the critical groups of directed hinge graphs? In particular, how does the critical group behave when all base shapes have the same orientation?

\end{enumerate}
    \vspace{0.25cm}


\section*{Acknowledgements}

The authors thank Jackson Morrow for introducing them to one another.
The authors also thank Carlos Alfaro, Nathan Kaplan, and Ralihe Villagr\'an for fruitful correspondence. 
A special thank you to Dave Jensen for his time and support in reading an earlier draft of this paper and his insightful feedback and to the anonymous referee whose comments helped refine our exposition.


\bibliographystyle{amsplain}
\bibliography{bibliography}

\end{document}